\documentclass{article}
\usepackage{amsmath}
\usepackage{amssymb}
\usepackage{amsthm}
\usepackage{tikz}
\usepackage{hyperref}

\numberwithin{equation}{section}
\newtheorem{theorem}[equation]{Theorem}
\newtheorem{lemma}[equation]{Lemma}
\newtheorem{corollary}[equation]{Corollary}

\newcommand{\ds}{\displaystyle}

\title{The relative isoperimetric profile of a unit square with a central square removed}
\author{Kyle Byassee, Hunter Johnson, Brandon Jones, Ezra Nance}
\date{}

\begin{document}
\maketitle
\begin{abstract}
    We calculate the exact values of the relative isoperimetric profile for a unit square with a central square removed. In other words, given a removed central square of a fixed size, we provide a function which produces the smallest relative perimeter which bounds any possible area within this shape.
\end{abstract}

\section{Introduction}
Let $Q\subseteq \mathbb{R}^2$ be a measurable subset with non-zero measure, $|Q|$. For any measurable subset $S\subseteq Q$, we let $P(S)$ denote the \textit{relative perimeter} of $S$, that is, the perimeter of $S$ which belongs to the interior of $Q$.  For each non-negative number $A\leq |Q|$, we define $f_{Q}(A) := \inf \{P(S): S\subseteq Q \text{ and } |S| = A\}$.  The function $f_Q$ is called the \textit{relative isoperimetric profile} of $Q$, and any subset $S$ which achieves this infimum is called an \textit{isoperimetric minimizer}. Note for any minimizer $S \subseteq Q$, $f_Q(|S|) = f_Q(|S'|)$, where $S' \subseteq Q$ is the complement of $S$. Since $f_Q$ has this symmetry, we will only ever consider the profile on half of its domain.

In \cite{BB} the relative isoperimetric profile was determined for $Q = [0,1]^2$. In \cite{DDNN}, this result was generalized, and $f_Q(A)$ was determined for the 1-parameter family $Q_a = [0,1]^2 \setminus [0,a)^2$. In this work, we generalize \cite{BB} in a different way by determining the relative isoperimetric profile for the 1-parameter family $Q_a~=~[0,1]^2 \setminus (\frac{1}{2}(1-a),\frac{1}{2}(1+a))^2$, see Figure \ref{Q_a}. In Section \ref{sec:MainTheorem} we provide the relative isoperimetric profile of $Q_a$ in Theorems \ref{theorem:Region_5}, \ref{theorem:Region_4}, \ref{theorem:Region_3}, \ref{theorem:Region_2}, and \ref{theorem:Region_1}. The profile is split up based on the value of $a \in (0,1)$.
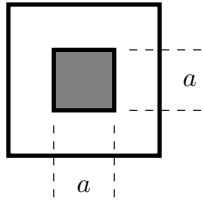
\begin{figure}[ht]
    \centering
    \begin{tikzpicture}[scale=2]
        \draw[ultra thick] (0,0) rectangle (1,1);
        \filldraw[ultra thick,fill=gray] (0.3,0.3) rectangle (0.7,0.7);
        \draw[dashed] (0.8, 0.7) -- (1.3,0.7);
        \draw[dashed] (0.8, 0.3) -- (1.3,0.3);
        \draw (1.2,0.5) node {$a$};
        \draw[dashed] (0.3, 0.2) -- (0.3,-0.3);
        \draw[dashed] (0.7, 0.2) -- (0.7,-0.3);
        \draw (0.5,-0.2) node {$a$};
    \end{tikzpicture}
    \caption{The 1-parameter family of shapes considered, $Q_a$}
    \label{Q_a}
\end{figure}

To find these isoperimetric minimizers, we must consider all possible subsets, $S \subseteq Q_a$. That is to say, our search space is infinite. Much work has been done to address this issue, both for general shapes in any dimension and for polygonal shapes in two dimensions. We have combined several results into the following theorem. The first bullet was proven by Gonzalez, Massari, and Tamanini \cite[Theorem 2, pg. 29]{GMT}, the second two were proven by Grüter \cite[Theorem (iii) pg. 264]{Gr}, and the last three were proven by DeVito, DeYeso, Nance, and Niedzialomski \cite[Theorem 1.1 and Proposition 2.2]{DDNN}.

\begin{theorem}\label{thm:backgroundProperties}
    Let $Q \subseteq \mathbb{R}^2$ be a bounded polygonal domain, and let $S \subseteq Q$ be a relative isoperimetric minimizer. Then
    \begin{itemize}
        \item $\partial S \cap \operatorname{int}(Q)$ is an embedded analytic submanifold of $Q$,
        \item $\partial S \cap \operatorname{int}(Q)$ is a disjoint union of curves, all of which have the same constant mean curvature,
        \item $\partial S$ intersects the non-corner points of $Q$ orthogonally,
        \item $\partial S$ never intersects a convex corner of $Q$,
        \item $\partial S$ intersects any non-convex corner of $Q$ at most one time,
        \item $\partial S$ has finitely many components.
    \end{itemize}
\end{theorem}

Throughout Section \ref{sec:region_elimination}, we utilize these results to reduce our search space down to something manageable. In so doing, we are able to eliminate all possible subsets of $Q_a$ from being minimizers other than six collections of subsets, see Figure \ref{fig:sixMinRegions}. In Section \ref{sec:MinRegions}, we discuss the properties of these six collections of regions, with special attention given to two region types which are bound by two arcs. Then in Section \ref{sec:MainTheorem}, we provide the minimums of the six region types.

\section{Region Elimination}\label{sec:region_elimination}
In this section we focus on eliminating region types from being isoperimetric minimizers. We begin with a proof that for any region $S \subseteq Q_a$, the boundary $\partial S$ is composed of at most $2$ components. 

\subsection{Easy Bounds}
\begin{lemma}\label{lem:boundedBy1}
    For all $a \in (0,1)$, the isoperimetric profile function $f_a$ is bounded above by $1$.
\end{lemma}
\begin{proof}
    Consider an area $A \in \left[0,\frac{1}{2}(1-a)\right)$.
    Then we can bound this area with a vertical line segment of length 1.
    Now consider an area,
    \[A \in \left[\frac{1}{2}(1-a), \frac{1}{2}(1-a^2)\right].\]
    This area can be bound by two vertical line segments, each with a length of $\frac{1}{2}(1-a)$.
    Thus, the area can be bound with a perimeter of $1-a$, which is less than $1$.
\end{proof}

\begin{lemma}\label{lem:boundedBy1MinusA}
    For all $a \in [\frac{1}{3},1)$, the isoperimetric profile function $f_a$ is bounded above by $1-a$.
\end{lemma}
\begin{proof}
    Note that the largest quarter circle that can fit in the corner of $Q_a$ has an area of $\frac{\pi}{8}(1-a)^2$. Now consider an area $A \in \left[0,\frac{1}{\pi}(1-a)^2\right]$. Then bounding this area with a quarter circle gives us a perimeter of $\sqrt{\pi A}$. In other words, the perimeter is at most $\sqrt{\pi \cdot \frac{1}{\pi}(1-a)^2} = 1-a$. 

    Consider regions whose boundary consists of two line segments which orthogonally connect adjacent sides of the central square to the edges of the outer square. These regions can have an area $\frac{1}{4}(1-a)^2 \leq A \leq \frac{1}{4}(1-a)^2 + a(1-a)$. Now consider an area $A \in \left(\frac{1}{\pi}(1-a)^2, \frac{1}{4}(1-a)^2 + a(1-a)\right]$. Since $\frac{1}{\pi}(1-a)^2 > \frac{1}{4}(1-a)^2$, we know that an area of $A$ can be realized by an adjacent line region, which has a perimeter of $1-a$.

    Consider regions whose boundary consists of two line segments which orthogonally connect opposite sides of the central square to the edges of the outer square. These regions can have an area $\frac{1}{2}(1-a) \leq A \leq \frac{1}{2}|Q_a|$.
    Let $A \in \left(\frac{1}{4}(1-a)^2 + a(1-a),\frac{1}{2}|Q_a|\right]$. Note that $\frac{1}{4}(1-a)^2 + a(1-a) \geq \frac{1}{2}(1-a)$ if and only if $a \geq \frac{1}{3}$. Thus, the area $A$ can be realized by an opposite lines region, which has a perimeter of $1-a$.
\end{proof}

\subsection{Regions with 1 Boundary Component}
We will now consider all possible regions which are bound by a single component. Regions whose boundary is a single line segment will not by eliminated, and thus will appear as a possible minimizer. Therefore, we will restrict our attention to regions whose boundary is a single arc of a circle. To help with this we will first establish the following bound on the relative perimeter.

\begin{lemma}\label{lem:quarterCircleBound}
    Let $S \subseteq Q_a$ be a region whose boundary is a single arc of a circle which follows the conditions of Theorem \ref{thm:backgroundProperties}. Then $P(S) \geq \sqrt{\pi|S|}$, with equality if and only if $S$ is a quarter-circle region.
\end{lemma}
\begin{proof}
    Since $S \subseteq Q_a$ follows the conditions set in Theorem \ref{thm:backgroundProperties}, the proof reduces to checking the following cases.
    
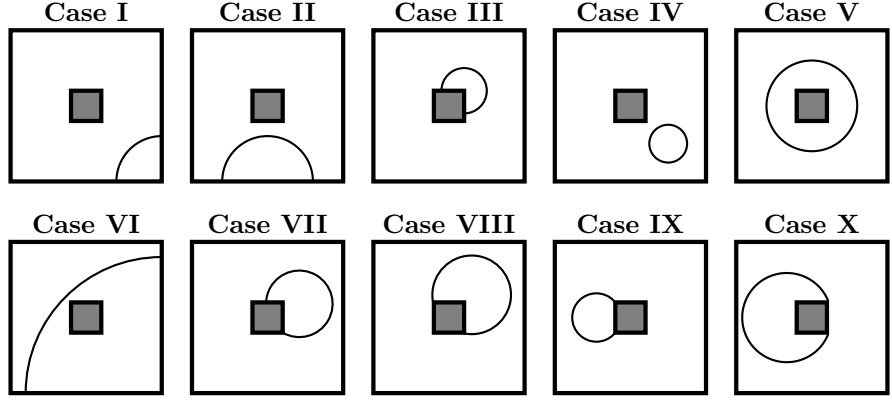
\begin{center}
        \begin{tikzpicture}[scale=2]
            \begin{scope}[shift={(0,0)}]
                \draw[ultra thick] (0,0) rectangle (1,1);
                \filldraw[ultra thick,fill=gray] (0.4,0.4) rectangle (0.6,0.6);
                \draw[thick] (0.7,0) arc (180:90:0.3);
                \draw (0.5, 1) node[anchor=south] {\textbf{Case I}};
            \end{scope}
            \begin{scope}[shift={(1.2,0)}]
                \draw[ultra thick] (0,0) rectangle (1,1);
                \filldraw[ultra thick,fill=gray] (0.4,0.4) rectangle (0.6,0.6);
                \draw[thick] (0.8,0) arc (0:180:.3);
                \draw (0.5, 1) node[anchor=south] {\textbf{Case II}};
            \end{scope}
            \begin{scope}[shift={(2.4,0)}]
                \draw[ultra thick] (0,0) rectangle (1,1);
                \filldraw[ultra thick,fill=gray] (0.4,0.4) rectangle (0.6,0.6);
                \draw[thick] (.6,.45) arc (-90:180:.15);
                \draw (0.5, 1) node[anchor=south] {\textbf{Case III}};
            \end{scope}
            \begin{scope}[shift={(3.6,0)}]
                \draw[ultra thick] (0,0) rectangle (1,1);
                \filldraw[ultra thick,fill=gray] (0.4,0.4) rectangle (0.6,0.6);
                \draw[thick] (0.75,0.25) circle (.125);
                \draw (0.5, 1) node[anchor=south] {\textbf{Case IV}};
            \end{scope}
            \begin{scope}[shift={(4.8,0)}]
                \draw[ultra thick] (0,0) rectangle (1,1);
                \filldraw[ultra thick,fill=gray] (0.4,0.4) rectangle (0.6,0.6);
                \draw[thick] (0.5,0.5) circle (0.3);
                \draw (0.5, 1) node[anchor=south] {\textbf{Case V}};
            \end{scope}
            \begin{scope}[shift={(0,-1.4)}]
                \draw[ultra thick] (0,0) rectangle (1,1);
                \filldraw[ultra thick,fill=gray] (0.4,0.4) rectangle (0.6,0.6);
                \draw[thick] (0.1,0) arc (180:90:0.9);
                \draw (0.5, 1) node[anchor=south] {\textbf{Case VI}};
            \end{scope}
            \begin{scope}[shift={(1.2,-1.4)}]
                \draw[ultra thick] (0,0) rectangle (1,1);
                \filldraw[ultra thick,fill=gray] (0.4,0.4) rectangle (0.6,0.6);
                \draw[thick] (.6,.4) arc (-120:180:.22);
                \draw (0.5, 1) node[anchor=south] {\textbf{Case VII}};
            \end{scope}
            \begin{scope}[shift={(2.4,-1.4)}]
                \draw[ultra thick] (0,0) rectangle (1,1);
                \draw[thick] (0.65,0.65) circle (0.26);
                \filldraw[ultra thick,fill=gray] (0.4,0.4) rectangle (0.6,0.6);
                \draw (0.5, 1) node[anchor=south] {\textbf{Case VIII}};
            \end{scope}
            \begin{scope}[shift={(3.6,-1.4)}]
                \draw[ultra thick] (0,0) rectangle (1,1);
                \draw[thick] (0.275,0.5) circle (0.16);
                \filldraw[ultra thick,fill=gray] (0.4,0.4) rectangle (0.6,0.6);
                \draw (0.5, 1) node[anchor=south] {\textbf{Case IX}};
            \end{scope}
            \begin{scope}[shift={(4.8,-1.4)}]
                \draw[ultra thick] (0,0) rectangle (1,1);
                \draw[thick] (0.608,0.61) arc (22:338:0.295);
                \filldraw[ultra thick,fill=gray] (0.4,0.4) rectangle (0.6,0.6);
                \draw (0.5, 1) node[anchor=south] {\textbf{Case X}};
            \end{scope}
        \end{tikzpicture}
    \end{center}
    For cases I-VI, a simple computation gives us the desired result.
    \begin{itemize}
        \item \textbf{Case I:}
        \[
            P(S) = \sqrt{\pi |S|}.
        \]
        \item \textbf{Case II:} 
        \[
            P(S) = \sqrt{2\pi |S|} > \sqrt{\pi|S|}.
        \]
        \item \textbf{Case III:}
        \[
            P(S) = \sqrt{3\pi |S|} > \sqrt{\pi |S|}.
        \]
        \item \textbf{Case IV:}
        \[
            P(S) = 2\sqrt{\pi |S|} > \sqrt{\pi |S|}.
        \]
        \item \textbf{Case V:}
        \[
            P(S) = 2\sqrt{\pi \left(|S| + a^2\right)} > \sqrt{\pi |S|}.
        \]
        \item \textbf{Case VI:}
        \[
            P(S) = \sqrt{\pi \left(|S|+a^2\right)} > \sqrt{\pi |S|}.
        \]
    \end{itemize}
    We can describe cases VII, VIII, and IX with the following figure.
    
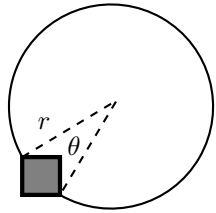
\begin{center}
        \begin{tikzpicture}[scale=5]
            \filldraw[ultra thick,fill=gray] (0.45,0.45) rectangle (0.55,0.55);
            \draw[thick] (0.55,0.45) arc (-135+15:225-15:0.27);
            \draw[thick,dashed] (0.45,0.55) -- (0.7,0.7) -- (0.55,0.45);
            \draw (0.55,0.6) node[anchor=south east] {$r$};
            \draw (0.63,0.63) node[anchor=north east] {$\theta$};
        \end{tikzpicture}
    \end{center}
    \begin{itemize}
        \item \textbf{Cases VII/VIII:} Note that $|S| < \pi r^2$, which implies $r > \sqrt{\frac{1}{\pi}|S|}$. Also, in both cases $\theta < \frac{\pi}{2}$. Hence, $2\pi -\theta > \frac{3\pi}{2}$. Putting this together, we have
        \[
            P(S) = (2\pi - \theta)r > \frac{3\pi}{2}\sqrt{\frac{1}{\pi}|S|} > \sqrt{\pi |S|}.
        \]
        \item \textbf{Case IX:} We still have that $|S| < \pi r^2$, which implies $r > \sqrt{\frac{1}{\pi}|S|}$. However, in this case $\theta < \pi$. Hence, $2\pi - \theta > \pi$. Thus,
        \[
            P(S) = (2\pi - \theta)r > \pi\sqrt{\frac{1}{\pi}|S|} = \sqrt{\pi|S|}.
        \]
    \end{itemize}
    We can describe case X with a similar figure, below.
    
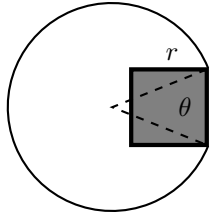
\begin{center}
        \begin{tikzpicture}[scale=5]
        \draw[thick] (0.6035,0.6035) arc (22:338:0.275);
            \filldraw[ultra thick,fill=gray] (0.4,0.4) rectangle (0.6,0.6);
            \draw[thick,dashed] (0.6,0.6) -- (0.35,0.5) -- (0.6,0.4);
            \draw (0.55,0.6) node[anchor=south east] {$r$};
            \draw (0.5,0.5) node[anchor=west] {$\theta$};
        \end{tikzpicture}
    \end{center}
    \begin{itemize}
        \item \textbf{Case X:} In this case we still have that $r > \sqrt{\frac{1}{\pi}|S|}$. Also, the max value of $\theta$ occurs when the center of the circle aligns with the center of the square. Hence, $\theta < \frac{\pi}{2}$. Thus, we have
        \[
            P(S) = (2\pi -\theta)r > \frac{3\pi}{2}\sqrt{\frac{1}{\pi}|S|} > \sqrt{\pi |S|}.
        \]
    \end{itemize}
\end{proof}

With this bound in place, we are now ready to say that minimizing regions bound by a single curve have a boundary which is either a line segment or a quarter of a circle. Since we have not eliminated the possibility that a line segment bounded region is possible, it suffices to show the following lemma.

\begin{lemma}\label{lem:QCirclesMinimizers}
    Let $S \subseteq Q_a$ be a minimizing region whose boundary is a single arc of a circle. Then $S$ must be a quarter-circle region.
\end{lemma}
\begin{proof}
    Let us assume that $S$ is not a quarter-circle region. Since $S$ is a minimizer, we know that it satisfies Theorem \ref{thm:backgroundProperties}. Thus,
    from Lemma \ref{lem:quarterCircleBound} we have that
    \[
        P(S) > \sqrt{\pi |S|}.
    \]
    
    If $|S| \leq \frac{\pi}{8}(1-a)^2$, then we can fit a quarter-circle region $S' \subseteq Q_a$ such that $|S'| = |S|$ and $P(S') = \sqrt{\pi|S|}$. However, this contradicts $S$ being an isoperimetric minimizer. Thus, for the remainder of the proof we will assume that $|S| > \frac{\pi}{8}(1-a)^2$. From this lower bound on area we have
    \[
        P(S) > \sqrt{\pi |S|} > \frac{\pi}{2\sqrt{2}}(1-a) > 1-a.
    \]
    If $a \geq \frac{1}{3}$, then by Lemma $\ref{lem:boundedBy1MinusA}$ $P(S) \leq 1-a$, which is a contradiction.

    We will now split into the same cases as in Lemma \ref{lem:quarterCircleBound}. Note that by Lemma \ref{lem:boundedBy1} we have $P(S) \leq 1$. Thus, if we can show that $P(S) > 1$, then we are done. By our previous arguments, it suffices to show this result under the assumption that $|S| > \frac{\pi}{8}(1-a)^2$ and $a < \frac{1}{3}$.
    \begin{itemize}
        \item \textbf{Case I:} By assumption, $S$ is not a quarter-circle region.
        \item \textbf{Case II:}
        \[
            P(S) = \sqrt{2\pi|S|} > \frac{\pi}{2}(1-a) > \frac{\pi}{3} > 1.
        \]
        \item \textbf{Case III:}
        \[
            P(S) = \sqrt{3\pi|S|} > \frac{\pi\sqrt{3}}{2\sqrt{2}}(1-a) > \frac{\pi}{\sqrt{6}} > 1.
        \]
        \item \textbf{Case IV:}
        \[
            P(S) = 2\sqrt{\pi|S|} > \frac{\pi}{\sqrt{2}}(1-a) > \frac{2\pi}{3\sqrt{2}} > 1.
        \]
        \item \textbf{Case V:}
        \[
            P(S) = 2\sqrt{\pi\left(|S|+a^2\right)} > 2\sqrt{\pi|S|}> \frac{\pi}{\sqrt{2}}(1-a) > \frac{2\pi}{3\sqrt{2}} > 1.
        \]
    \end{itemize}
    For the remainder of the cases, we must consider the radius of the arc.
    \begin{itemize}
        \item \textbf{Case VI:} From the geometry, in this case we have that $r > \frac{1+a}{\sqrt{2}}$. Thus, we have
        \[
            P(S) = \frac{1}{2}\pi r > \frac{\pi(1+a)}{2\sqrt{2}} > \frac{\pi}{2\sqrt{2}} > 1.
        \]
        \item \textbf{Cases VII/VIII:} As in Lemma \ref{lem:quarterCircleBound}, we have $r > \sqrt{\frac{1}{\pi}|S|}$ and $2\pi - \theta > \frac{3\pi}{2}$. Hence,
        \[
            P(S) = (2\pi -\theta)r > \frac{3\pi}{2}\sqrt{\frac{1}{\pi}|S|} > \frac{3\pi}{4\sqrt{2}}(1-a) > \frac{\pi}{2\sqrt{2}} > 1.
        \]
        \item \textbf{Case IX:} Note that $|S| = \pi r^2 - \frac{1}{2}r^2(\theta-\sin\theta)$. Now solving for the radius and using our bounds on $|S|$ and $a$, we have
        \[
            r = \sqrt{\frac{2|S|}{2\pi - \theta + \sin\theta}} > \frac{1-a}{2}\sqrt{\frac{\pi}{2\pi - \theta + \sin\theta}} > \frac{1}{3}\sqrt{\frac{\pi}{2\pi - \theta + \sin\theta}}.
        \]
        Thus,
        \[
            P(S) = (2\pi - \theta)r > \frac{1}{3}\sqrt{\frac{\pi(2\pi -\theta)^2}{2\pi - \theta + \sin\theta}}.
        \]
        Hence, this case is finished if
        \[
            \sqrt{\frac{\pi(2\pi -\theta)^2}{2\pi - \theta + \sin\theta}} > 3.
        \]
        Note that in this case $0 < \theta < \pi$. Thus, $2\pi -\theta +\sin\theta > 0$. Therefore, we have
        \[
            \sqrt{\frac{\pi(2\pi -\theta)^2}{2\pi - \theta + \sin\theta}} > \pi \hspace{1em} \iff \hspace{1em} (\pi - \theta)(2\pi - \theta) - \pi \sin\theta > 0.
        \]
        To see that the right inequality is true, consider the change of variables to $x = \pi - \theta$ with $0 < x < \pi$, and note that $x > \sin x$ in this interval. Putting these together, we get
        \[
            (\pi - \theta)(2\pi - \theta) - \pi \sin\theta = x(\pi + x) - \pi\sin x > x(\pi + x) - \pi x = x^2 > 0.
        \]
        \item \textbf{Case X:} In this case we have $r > \sqrt{\frac{1}{\pi}|S|}$ and $2\pi -\theta > \frac{3\pi}{2}$. Hence,
        \[
            P(S) = (2\pi -\theta)r > \frac{3\pi}{2}\sqrt{\frac{1}{\pi}|S|} > \frac{3\pi}{4\sqrt{2}}(1-a) > \frac{\pi}{2\sqrt{2}} > 1.
        \]
    \end{itemize}
\end{proof}

\subsection{Regions with 2 Boundary Components}
We now consider all possible minimizing regions which are bound by two components. We begin by considering regions bound by line segments. Note that by Lemma \ref{lem:boundedBy1}, the boundary of a minimizing region cannot contain two lines of length $1$. Hence, any possible two-line region will consist of segments of length $\frac{1}{2}(1-a)$. The following lemma characterizes these minimizers.

\begin{lemma}\label{lemma:SPE}
    Let $S \subseteq Q_a$ be a minimizer with $\partial S$ consisting of two line segments. Then $S$ can only be an Adjacent Lines region or an Opposite Lines region.
\end{lemma}
\begin{proof}
    Since $S$ is a minimizer, its boundary must intersect $Q_a$ orthogonally. This leaves three possibilities for $\partial S$: either the two line segments are on the same side, on adjacent sides, or on opposite sides. Assume $\partial S$ consists of two line segments on the same side, see Figure~\ref{sameSidePair}.
    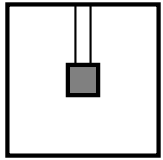
\begin{figure}[ht]
        \centering
        \begin{tikzpicture}[scale=2]
            \draw[ultra thick] (0,0) rectangle (1,1);
            \filldraw[ultra thick,fill=gray] (0.4,0.4) rectangle (0.6,0.6);
            \draw[thick] (0.45,0.6) -- (0.45,1);
            \draw[thick] (0.55,0.6) -- (0.55,1);
        \end{tikzpicture}
        \caption{Same Side Lines}
        \label{sameSidePair}
    \end{figure}
    
    Then $0 \leq |S| \leq \frac{1}{2}a(1-a)$. From here we will split into cases. First, assume that $a < \frac{2}{2+\pi}$. Then $\frac{1}{2}a(1-a) < \frac{1}{\pi}(1-a)^2$. Now consider a quarter-circle region $S'$ with $|S| = |S'|$. Then
    \[
        P(S') = \sqrt{\pi |S'|} = \sqrt{\pi |S|} < 1-a = P(S).
    \]
    This contradicts $S$ being a minimizer.

    Now assume that $a > \frac{1}{3}$. Within the proof of Lemma \ref{lem:boundedBy1MinusA} we achieve a relative perimeter of $1-a$ using only adjacent line regions and opposite line regions. Hence, we never have to consider regions like Figure~\ref{sameSidePair}.
\end{proof} 

Next we will consider all regions which are bound by two arcs of circles. Figure~\ref{fig:typesofarcs} displays these regions. The following lemmas will characterize minimizing regions of this form.
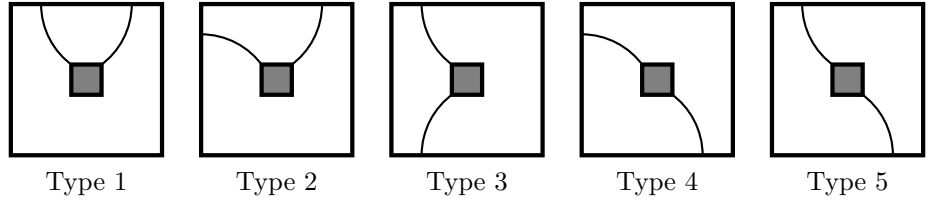
\begin{figure}[ht]
    \centering
    \begin{tikzpicture}[scale=2]
        \draw[ultra thick] (0,0) rectangle (1,1);
        \filldraw[ultra thick,fill=gray] (0.4,0.4) rectangle (0.6,0.6);
        \draw[thick] (0.2,1) arc (180:235:0.5);
        \draw[thick] (0.8,1) arc (0:-55:0.5);
        \draw (0.5,-0.05) node[anchor=north] {Type 1};
    \end{tikzpicture}
    \hfill
    \begin{tikzpicture}[scale=2]
        \draw[ultra thick] (0,0) rectangle (1,1);
        \filldraw[ultra thick,fill=gray] (0.4,0.4) rectangle (0.6,0.6);
        \draw[thick] (0,0.8) arc (90:35:0.5);
        \draw[thick] (0.8,1) arc (0:-55:0.5);
        \draw (0.5,-0.05) node[anchor=north] {Type 2};
    \end{tikzpicture}
    \hfill
    \begin{tikzpicture}[scale=2]
        \draw[ultra thick] (0,0) rectangle (1,1);
        \filldraw[ultra thick,fill=gray] (0.4,0.4) rectangle (0.6,0.6);
        \draw[thick] (0.2,1) arc (180:235:0.5);
        \draw[thick] (0.2,0) arc (180:125:0.5);
        \draw (0.5,-0.05) node[anchor=north] {Type 3};
    \end{tikzpicture}
    \hfill
    \begin{tikzpicture}[scale=2]
        \draw[ultra thick] (0,0) rectangle (1,1);
        \filldraw[ultra thick,fill=gray] (0.4,0.4) rectangle (0.6,0.6);
        \draw[thick] (0,0.8) arc (90:35:0.5);
        \draw[thick] (0.8,0) arc (0:55:0.5);
        \draw (0.5,-0.05) node[anchor=north] {Type 4};
    \end{tikzpicture}
    \hfill
    \begin{tikzpicture}[scale=2]
        \draw[ultra thick] (0,0) rectangle (1,1);
        \filldraw[ultra thick,fill=gray] (0.4,0.4) rectangle (0.6,0.6);
        \draw[thick] (0.2,1) arc (180:235:0.5);
        \draw[thick] (0.8,0) arc (0:55:0.5);
        \draw (0.5,-0.05) node[anchor=north] {Type 5};
    \end{tikzpicture}
    \caption{Types of regions bounded by two circular arcs}
    \label{fig:typesofarcs}
\end{figure}

\begin{lemma}\label{lemma:DARE}
    Let $S\subseteq Q_a$ be a Type 2 Region or a Type 5 Region.
    Then $S$ is not a minimizer.
\end{lemma}
\begin{proof}
    Suppose that $S\subseteq Q_a$ with $\partial S$ containing arcs of circles. Recall from Theorem~\ref{thm:backgroundProperties} that if $S$ is a minimizer, then $\partial S$ must have constant mean curvature.
    We now consider the following cases.
    \begin{itemize}
        \item \textbf{Case 1:} Suppose that $S$ is a Type 2 Region.
        Then the region $S$ can only bound an area, $A=\frac{1}{2}(1-a)\cdot\left(\frac{1}{2}(1-a)+a\right)$ and this area can be bounded by line segments extending orthogonally from the corners of the central square to the outer square. See the below figure.
        
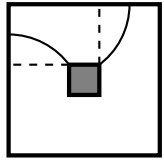
\begin{center}
            \begin{tikzpicture}[scale=2]
                \draw[ultra thick] (0,0) rectangle (1,1);
                \filldraw[ultra thick,fill=gray] (0.4,0.4) rectangle (0.6,0.6);
                \draw[thick] (0,0.8) arc (90:35:0.5);
                \draw[thick] (0.8,1) arc (0:-55:0.5);
                \draw[thick,dashed] (0.4,0.6) -- (0,0.6);
                \draw[thick,dashed] (0.6,0.6) -- (0.6,1);
            \end{tikzpicture}
        \end{center}
        
        Additionally, we have that the rectangular region has a perimeter of $1-a$ and the perimeter of the Type 2 Region is greater than $1-a$.
        Thus Type 2 Regions are not minimizers.
        
        \item \textbf{Case 2:} Suppose that $S$ is a Type 5 region.
        We can make a similar argument as in Case 1. Since a Type 5 Region can only bound an area, $A=\frac{1}{2}|Q_a|$ then we can bound this area by parallel line segments extending from the corners of the central square to the outer square. See the below figure.
        
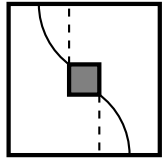
\begin{center}
            \begin{tikzpicture}[scale=2]
                \draw[ultra thick] (0,0) rectangle (1,1);
                \filldraw[ultra thick,fill=gray] (0.4,0.4) rectangle (0.6,0.6);
                \draw[thick] (0.2,1) arc (180:235:0.5);
                \draw[thick] (0.8,0) arc (0:55:0.5);
                \draw[thick,dashed] (0.4,0.6) -- (0.4,1);
                \draw[thick,dashed] (0.6,0.4) -- (0.6,0);
            \end{tikzpicture}
        \end{center}
        
        Since the perimeter of the Type 5 Region is greater than that of the region bounded by lines, then Type 5 Regions are not minimizers.
    \end{itemize}
\end{proof}

\begin{lemma}\label{lem:no_horns}
    Let $S \subseteq Q_a$ be a minimizer with $\partial S$ two circular arcs. Then $S$ cannot be a Type 1 Region.
\end{lemma}
\begin{proof}
    Consider the case when $\frac{1}{3}\leq a<1$.
    Let $S \subseteq Q_a$ be a Type 1 Region as in the figure below.
    
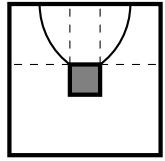
\begin{center}
        \begin{tikzpicture}[scale=2]
            \draw[ultra thick] (0,0) rectangle (1,1);
            \filldraw[ultra thick,fill=gray] (0.4,0.4) rectangle (0.6,0.6);
            \draw[thick] (0.2,1) arc (180:235:0.5);
            \draw[thick] (0.8,1) arc (0:-55:0.5);
            \draw[dashed] (0.4,0.6) -- (0.4,1);
            \draw[dashed] (0.6,0.6) -- (0.6,1);
            \draw[dashed] (0.4,0.6) -- (0,0.6);
            \draw[dashed] (0.6,0.6) -- (1,0.6);
        \end{tikzpicture}
    \end{center}
    Then we have the following bounds on the area of $S$,
    \[\frac{a}{2}(1-a)<|S|<\frac{1}{2}(1-a).\]
    Now consider a region bounded by line segments extending orthogonally from adjacent sides of the central square to the outer square. See the below figure.
    
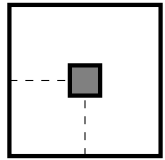
\begin{center}
        \begin{tikzpicture}[scale=2]
             \draw[ultra thick] (0,0) rectangle (1,1);
            \filldraw[ultra thick,fill=gray] (0.4,0.4) rectangle (0.6,0.6);
            \draw[dashed] (0.5,0) -- (0.5,0.4);
            \draw[dashed] (0,0.5) -- (0.4,0.5);
        \end{tikzpicture}
    \end{center}
    We can determine that the perimeter of this region is $1-a$.
    Additionally the area enclosed in this region, $|S'|$ is bounded by
    \[\left(\frac{1}{2}(1-a)\right)^2\leq |S'|\leq\left(\frac{1}{2}(1-a)\right)^2+a(1-a).\]
    Then given our assumption on $a$, we can determine that
    \[\left(\frac{1}{2}(1-a)\right)^2\leq\frac{a}{2}(1-a)<|S|<\frac{1}{2}(1-a)\leq\left(\frac{1}{2}(1-a)\right)^2+a(1-a).\]
    Thus when $\frac{1}{3}\leq a<1$, we can replace $S$ with $S'$ such that $|S| = |S'|$. Furthermore, $P(S) > P(S')$. Hence, $S$ is not a minimizer.

    Based on the geometry of Type 1 regions, we can express their area and perimeter as follows:
    \[
        A_1(\theta,a) = \left(\frac{1-a}{2}\right)^2\left[\theta\csc^2\theta - \cot\theta\right] + a\left(\frac{1-a}{2}\right),
    \]
    and
    \[
        P_1(\theta,a) = (1-a)\theta\csc\theta,
    \]
    where $\theta \in (0,\frac{\pi}{2})$ is the angle shown in the figure below.

    \begin{center}
        \begin{tikzpicture}[scale=4]
            \draw[ultra thick] (0,0) rectangle (1,1);
            \filldraw[ultra thick,fill=gray] (0.4,0.4) rectangle (0.6,0.6);
            \draw[thick] (0.2,1) arc (180:234:0.5);
            \draw[thick] (0.8,1) arc (0:-54:0.5);
            \draw[thick,dashed] (0.4,0.6) -- (0.7,1);
            \draw[thick,dashed] (0.6,0.6) -- (0.3,1);
            \draw (0.55,1) arc (180:235:0.15);
            \draw (0.45,1) arc (360:305:0.15);
            \draw (0.13,1) -- (0.13,0.925) -- (0.205,0.925);
            \draw (0.87,1) -- (0.87,0.925) -- (0.795,0.925);
            \draw (0.5,0.95) node[anchor=north] {$\theta$};
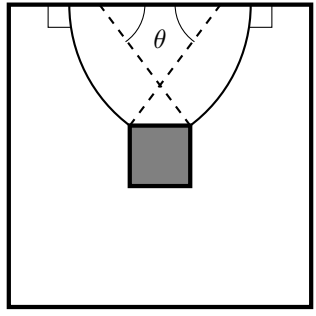
            \draw (0.5,-0.05) node[anchor=north] {Type 1 Region};
        \end{tikzpicture}
    \end{center}
    Note that for all fixed $a \in (0,1)$, the function $A_1: (0,\frac{\pi}{2}) \to \mathbb{R}$ is injective. We can see this by taking its derivative:
    \[
        \frac{dA_1}{d\theta} = \frac{(1-a)^2\csc^2\theta(\theta\cot\theta-1)}{2}.
    \]
    Since $\tan\theta > \theta$ when $\theta \in (0,\frac{\pi}{2})$, we have that $\theta\cot\theta - 1 < 0$, making the derivative always negative. Hence, $A_1$ as a function of $\theta$ is monotonic. Thus, the perimeter of a Type 1 region is implicitly a function of its enclosed area.

    Let $a$ be a fixed value in $(0,\frac{1}{3})$, and $|S| \leq \frac{\pi}{8}(1-a)^2$. With this amount of area, we can replace $S$ with a quarter circle region $S'$ such that $|S| = |S'|$. It suffices to show that $P(S') \leq P(S)$. This inequality is true if and only if $\sqrt{\pi |S'|} \leq P_1(S)$. After squaring both sides and rearranging, we need
    \begin{equation}\label{eq:QC-better}
        P_1(S)^2 - \pi |S| \geq 0
    \end{equation}
    to be true. Expressing these quantities in terms of $\theta$ we need
    \[
        (1-a)^2\theta^2\csc^2\theta - \pi\left(\frac{1-a}{2}\right)^2\left[\theta\csc^2\theta - \cot\theta\right] + \pi a\left(\frac{1-a}{2}\right) \geq 0.
    \]
    to be true. For ease of notation, we define $F(\theta)$ to be the left-hand side of this inequality. Then taking its derivative, we have
    \[
        \frac{dF}{d\theta} = \frac{(\pi-4\theta)(1-a)^2\csc^2\theta(\theta\cot\theta-1)}{2}.
    \]
    From the analysis of $\frac{dA_1}{d\theta}$, the last three factors will result in a negative sign. Thus, the sign of $\frac{dF}{d\theta}$ is determined by $(\pi - 4\theta)$. Hence, for all $\theta \in (0,\frac{\pi}{4})$, $F$ is decreasing, and for all $\theta \in (\frac{\pi}{4},\frac{\pi}{2})$, $F$ in increasing. Therefore, the minimum value of $F$ occurs at $\theta = \frac{\pi}{4}$. Substituting into $F$ gives us $F(\frac{\pi}{4}) = \frac{\pi}{4}(1-a)(1-3a)$. Since $a \in (0,\frac{1}{3})$, $F(\frac{\pi}{4}) \geq 0$, and thus the inequality \eqref{eq:QC-better} is achieved.
    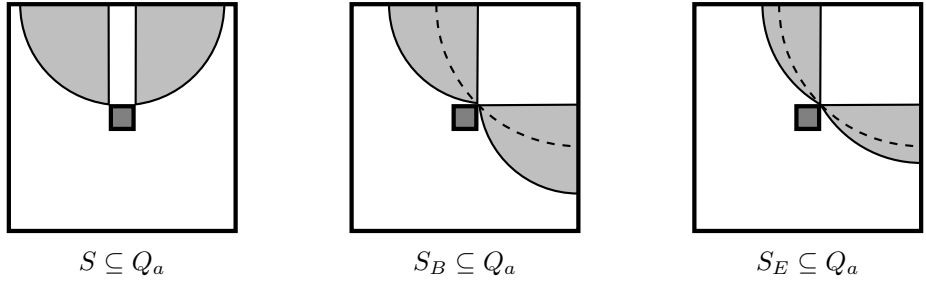
\begin{figure}[ht]
        \begin{tikzpicture}[scale=3]
            \filldraw[thick,fill=gray!50] (0.05,1) arc (180:263:0.445) -- (0.44,1);
            \filldraw[thick,fill=gray!50] (0.95,1) arc (0:-83:0.445) -- (0.56,1);
            \draw[ultra thick] (0,0) rectangle (1,1);
            \filldraw[ultra thick,fill=gray] (0.45,0.45) rectangle (0.55,0.55);
            \draw (0.5,-0.05) node[anchor=north] {$S \subseteq Q_a$};
        \end{tikzpicture}
        \hfill
        \begin{tikzpicture}[scale=3]
            \filldraw[thick,fill=gray!50] (0.165,1) arc (180:263.5:0.439) -- (0.557,1);
            \filldraw[thick,fill=gray!50] (1,0.165) arc (-90:-173.5:0.439) -- (1,0.557);
            \draw[dashed,thick] (0.3736,1) arc (180:270:0.6263);
            \draw[ultra thick] (0,0) rectangle (1,1);
            \filldraw[ultra thick,fill=gray] (0.45,0.45) rectangle (0.55,0.55);
            \draw (0.5,-0.05) node[anchor=north] {$S_B \subseteq Q_a$};
        \end{tikzpicture}
        \hfill
        \begin{tikzpicture}[scale=3]
            \filldraw[thick,fill=gray!50] (0.3,1) arc (180:240:0.51) -- (0.557,1);
            \filldraw[thick,fill=gray!50] (1,0.3) arc (-90:-150:0.51) -- (1,0.557);
            \draw[dashed,thick] (0.3736,1) arc (180:270:0.6263);
            \draw[ultra thick] (0,0) rectangle (1,1);
            \filldraw[ultra thick,fill=gray] (0.45,0.45) rectangle (0.55,0.55);
            \draw (0.5,-0.05) node[anchor=north] {$S_E \subseteq Q_a$};
        \end{tikzpicture}
        \caption{Equivalent area of Type 1 region with smaller perimeter}
        \label{fig:type1elimination}
    \end{figure}
    
    Finally, let $a$ be a fixed value in $(0,\frac{1}{3})$, and $|S| > \frac{\pi}{8}(1-a)^2$. Consider a region $S_B \subseteq Q_a$, where the two `arc regions' are positioned such that the arcs emanate from the same corner, see Figure \ref{fig:type1elimination}. Note that since $a < \frac{1}{3}$, we have that $a\left(\frac{1-a}{2}\right) < \left(\frac{1-a}{2}\right)^2$. Hence, $|S| < |S_B|$, while $P(S) = P(S_B)$. Now, reduce the arc length of each of the `arc regions.' This results in a smaller area and smaller perimeter. Reducing the arc length as much as possible results in two line segments. This would result in an area less than $\frac{\pi}{8}(1-a)^2$. Hence, there exists some region $S_E \subseteq Q_a$ such that $|S| = |S_E|$ and $P(S) > P(S_E)$, see Figure \ref{fig:type1elimination}. Thus, $S$ cannot be a minimizer.
\end{proof}

Therefore, by Lemmas \ref{lemma:DARE} and \ref{lem:no_horns}, if $S \subseteq Q_a$ is a minimizer with $\partial S$ consisting of two circular arcs, then $S$ can only be of Type 3 or Type 4.

\subsection{Number of Boundary Components}
To begin we need a result which is shown within the proof of Lemma \ref{lem:QCirclesMinimizers}.
\begin{corollary}\label{cor:bigArcArea}
    Let $a < \frac{1}{3}$ and $S \subseteq Q_a$ such that $|S| > \frac{\pi}{8}(1-a)^2$ and $\partial S$ is a single arc of a circle. Then $P(S) > 1$.
\end{corollary}
\begin{proof}
    Within the proof of Lemma \ref{lem:QCirclesMinimizers} we have exactly this result.
\end{proof}

\begin{lemma}\label{lem:replaceArcRegions}
    Let $a < \frac{1}{3}$ and $S \subseteq Q_a$ such that $\frac{\pi}{8}(1-a)^2 < |S| \leq \frac{1}{2}|Q_a|$. Additionally, let $\partial S$ consist of arcs of circles with the area of each component $|S_i| \leq \frac{\pi}{8}(1-a)^2$. Then there exists a region $S' \subseteq Q_a$ with the following properties:
    \begin{itemize}
        \item $|S| = |S'|$,
        \item $P(S) \geq P(S')$,
        \item $\partial S'$ consists of equally sized quarter circles.
    \end{itemize}
\end{lemma}
\begin{proof}
    Let $a \leq \frac{\pi - 2}{\pi+2}$. Then we have
    \[
        2\cdot \frac{\pi}{8}(1-a)^2 \geq \frac{1}{2}|Q_a| = \frac{1}{2}(1-a^2).
    \]
    In other words, the area $|S|$ will fit into two quarter circles. We will now partition $S$ into two areas:
    \[
        S = S_1 \cup S_2 \cup \cdots \cup S_n = T_1 \cup T_2,
    \]
    where $T_1 = S_1 \cup \cdots \cup S_k$ and $T_2 = S_{k+1} \cup \cdots \cup S_n$ such that $|T_1|, |T_2| \leq \frac{\pi}{8}(1-a)^2$. Now let $T_1'$ be a quarter-circle region with the same area as $T_1$, and let $T_2'$ be the same for $T_2$ in the opposite corner. Now by Lemma \ref{lem:quarterCircleBound} we have
    \[
        P(T_1) = \sum_{i=1}^k P(S_i) \geq \sum_{i=1}^k \sqrt{\pi |S_i|} \geq \sqrt{\pi \sum_{i=1}^k|S_i|} = \sqrt{\pi |T_1|} = P(T_1').
    \]
    The same holds true for $T_2'$. Let $S'$ be the region $T_1' \cup T_2'$ with the two quarter-circle regions adjusted so that they have the same radius while maintaining the same overall area. By Theorem \ref{thm:backgroundProperties}, this results in a perimeter which is equal to or smaller than the original. Thus,
    \[
        P(S) \geq P(T_1' \cup T_2') \geq P(S').
    \]
    Let $\frac{\pi-2}{\pi+2} < a < \frac{1}{3}$. Then we have
    \[
        3\cdot \frac{\pi}{8}(1-a)^2 \geq \frac{1}{2}|Q_a| = \frac{1}{2}(1-a^2).
    \]
    In other words, the area $|S|$ will fit into three quarter circles. We will now partition $S$ into three areas:
    \[
        S = S_1 \cup S_2 \cup \cdots \cup S_n = T_1 \cup T_2 \cup T_3,
    \]
    where $T_1 = S_1 \cup \cdots \cup S_k$, $T_2 = S_{k+1} \cup \cdots \cup S_\ell$, and $T_3 = S_{\ell+1} \cup \cdots \cup S_n$ such that $|T_1|, |T_2|, |T_3| \leq \frac{\pi}{8}(1-a)^2$. Now let $T_1'$ be a quarter-circle region with the same area as $T_1$, and let $T_2'$ and $T_3'$ be the same for $T_2$ and $T_3$, respectively. Now by Lemma \ref{lem:quarterCircleBound} we have
    \[
        P(T_1) = \sum_{i=1}^k P(S_i) \geq \sum_{i=1}^k \sqrt{\pi |S_i|} \geq \sqrt{\pi \sum_{i=1}^k|S_i|} = \sqrt{\pi |T_1|} = P(T_1').
    \]
    The same holds true for $T_2'$ and $T_3'$. Let $S'$ be the region $T_1' \cup T_2' \cup T_3'$ with the three quarter-circle regions adjusted so that they have the same radius while maintaining the same overall area. By Theorem \ref{thm:backgroundProperties}, this results in a perimeter which is equal to or smaller than the original. Thus,
    \[
        P(S) \geq P(T_1' \cup T_2' \cup T_3') \geq P(S').
    \]
\end{proof}

With Lemma \ref{lem:replaceArcRegions} in hand, we can now completely characterize the number of boundary components that minimizing regions have.

\begin{lemma}\label{lem:numberOfBoundaryComponents}
     Let $S \subseteq Q_a$ be a minimizer. Then we have the following:
    \begin{itemize}
        \item if $\partial S$ touches the inner and outer boundary of $Q_a$, then $\partial S$ consists of exactly 2 components,
        \item if $\partial S$ touches only the inner or outer boundary of $Q_a$, then $\partial S$ consists of exactly 1 component.
    \end{itemize}
\end{lemma}
\begin{proof}
    Let $S \subseteq Q_a$ be a minimizer. We will begin with the assumption that a component of $\partial S$ touches both the inner square and the outer boundary of $Q_a$. In order for $\partial S$ to bound a region, $\partial S$ must contain at least one more component which touches both the inner square and the outer boundary of $Q_a$. If $a \geq \frac{1}{3}$, then by Lemma \ref{lem:boundedBy1MinusA} we have $P(S) \leq 1-a$. However, since $\partial S$ consists of at least two components as described earlier, we have $P(S) \geq 2\cdot \frac{1}{2}(1-a) = 1-a$. This implies that $\partial S$ consists of exactly two components, each having length $\frac{1}{2}(1-a)$. 
    
    If on the other hand $a < \frac{1}{3}$, then by Lemma \ref{lem:boundedBy1} we have $P(S) \leq 1$. If $\partial S$ contains a third component which touches both the inner and outer square, then we have
    \[
        P(S) \geq 3 \cdot \frac{1}{2}(1-a) > \frac{3}{2}\cdot \frac{2}{3} = 1,
    \]
    a contradiction. Thus, if $\partial S$ has a third component, it must either only touch the inner square or only touch the outer square. If this third component is a line segment, then it must touch opposite sides of the outer square. In this case we have
    \[
        P(S) \geq 1 + 2\cdot \frac{1}{2}(1-a) > 1,
    \]
    a contradiction. Hence, we can assume that a third component of $\partial S$ is the arc of a circle. By Theorem \ref{thm:backgroundProperties}, this implies all components are arcs of circles with the same radius. In order for the arc of a circle to touch both the inner and outer squares, we must have that it's radius obeys the inequality $R > \frac{1}{2}(1-a)$. Furthermore, if the arc of a circle touches either only the inner square or only the outer square, then it must by at least a quarter circle. Thus, we have
    \[
        P(S) \geq \frac{1}{2}\pi R + 2\cdot \frac{1}{2}(1-a) > \left(\frac{\pi}{4} + 1\right)(1-a) > \frac{\pi + 4}{6} > 1,
    \]
    a contradiction. This concludes the proof of the first bullet point.
    
    For the second bullet point, we assume that all components of $\partial S$ either only touch the inner square or only touch the outer square. If one of the components is a line segment, then it must connect opposite sides of the outer square, and we have $P(S) \geq 1$. By Lemma \ref{lem:boundedBy1} we have $P(S) \leq 1$. Hence, in this case $\partial S$ must consist of exactly that single line segment.

    Next we will assume that $\partial S$ consists of several arcs of circles. Then by Lemma \ref{lem:quarterCircleBound}, we have
    \[
        P(S) = \sum P(S_i) \geq \sum \sqrt{\pi |S_i|} \geq \sqrt{\pi \sum |S_i|} = \sqrt{\pi |S|}.
    \]
    If $|S| \leq \frac{\pi}{8}(1-a)^2$, then we have $P(S) \leq \sqrt{\pi|S|}$. Hence, $\partial S$ would consist of exactly one quarter-circle. Thus, we will assume that $|S| > \frac{\pi}{8}(1-a)^2$. Substituting this into the inequality above, we have
    \[
        P(S) \geq \sqrt{\pi |S|} > \frac{\pi}{2\sqrt{2}}(1-a) > 1-a.
    \]
    By Lemma \ref{lem:boundedBy1MinusA}, if $a \geq \frac{1}{3}$ we have a contradiction.

    It remains to show that the lemma statement holds in the case where $|S| > \frac{\pi}{8}(1-a)^2$, $a < \frac{1}{3}$, and $\partial S$ consists of arcs of the same radius which touch either only the inner square or only the outer square. If any component of $\partial S$ bounds an area larger than $\frac{\pi}{8}(1-a)^2$, then $P(S) > 1$ by Corollary \ref{cor:bigArcArea}. Thus, we will assume that every component of $\partial S$ bounds an area of at most $\frac{\pi}{8}(1-a)^2$. If $a \leq \frac{\pi -2}{\pi+2}$, then by Lemma \ref{lem:replaceArcRegions} we have a two, same sized quarter-circle region $S'$ such that $|S|=|S'|$ and $P(S) \geq P(S')$. Thus,
    \[
        P(S) \geq P(S') = \sqrt{2 \pi |S|} > \frac{\pi}{2}(1-a) > \frac{\pi}{3} > 1.
    \]
    On the other hand, if $a > \frac{\pi-2}{\pi+2}$, then by Lemma \ref{lem:replaceArcRegions} we have a three, same-sized quarter-circle region $S''$ such that $|S| = |S''|$ and $P(S) \geq P(S'')$. Thus,
    \[
        P(S) \geq P(S'') = \sqrt{3 \pi |S|} > \frac{\pi\sqrt{3}}{2\sqrt{2}}(1-a) > \frac{\pi}{\sqrt{6}} > 1.
    \]
    This concludes the proof.
\end{proof}
    
\section{Minimizing Regions}\label{sec:MinRegions}
Utilizing Lemmas \ref{lemma:SPE}, \ref{lemma:DARE}, and \ref{lem:numberOfBoundaryComponents}, we have eliminated all but six regions from being minimizers. We now use this section to establish some properties of these remaining regions.
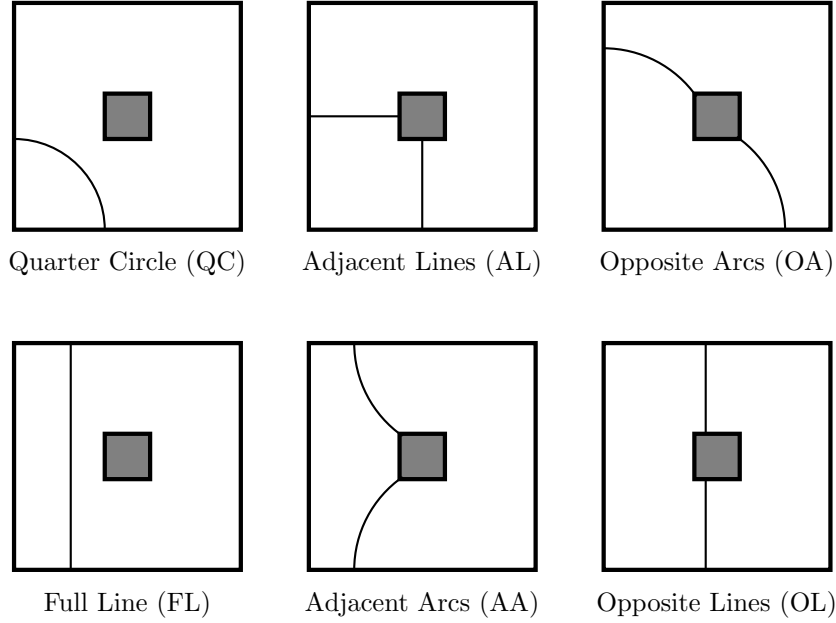
\begin{figure}[ht]
    \centering
    \begin{tikzpicture}[scale=3]
    \begin{scope}
        \draw[ultra thick] (0,0) rectangle (1,1);
        \filldraw[ultra thick,fill=gray] (0.4,0.4) rectangle (0.6,0.6);
        \draw[thick] (0.4,0) arc (0:90:0.4);
        \draw (0.5,-0.05) node[anchor=north] {Quarter Circle (QC)};
    \end{scope}
    \begin{scope}[shift={(1.3,0)}]
        \draw[ultra thick] (0,0) rectangle (1,1);
        \filldraw[ultra thick,fill=gray] (0.4,0.4) rectangle (0.6,0.6);
        \draw[thick] (0.5,0) -- (0.5,0.4);
        \draw[thick] (0,0.5) -- (0.4,0.5);
        \draw (0.5,-0.05) node[anchor=north] {Adjacent Lines (AL)};
    \end{scope}
    \begin{scope}[shift={(2.6,0)}]
        \draw[ultra thick] (0,0) rectangle (1,1);
        \filldraw[ultra thick,fill=gray] (0.4,0.4) rectangle (0.6,0.6);
        \draw[thick] (0,0.8) arc (90:35:0.5);
        \draw[thick] (0.8,0) arc (0:55:0.5);
        \draw (0.5,-0.05) node[anchor=north] {Opposite Arcs (OA)};
    \end{scope}
    \begin{scope}[shift={(0,-1.5)}]
        \draw[ultra thick] (0,0) rectangle (1,1);
        \filldraw[ultra thick,fill=gray] (0.4,0.4) rectangle (0.6,0.6);
        \draw[thick] (0.25,0) -- (0.25,1);
        \draw (0.5,-0.05) node[anchor=north] {Full Line (FL)};
    \end{scope}
    \begin{scope}[shift={(1.3,-1.5)}]
        \draw[ultra thick] (0,0) rectangle (1,1);
        \filldraw[ultra thick,fill=gray] (0.4,0.4) rectangle (0.6,0.6);
        \draw[thick] (0.2,1) arc (180:235:0.5);
        \draw[thick] (0.2,0) arc (180:125:0.5);
        \draw (0.5,-0.05) node[anchor=north] {Adjacent Arcs (AA)};
    \end{scope}
    \begin{scope}[shift={(2.6,-1.5)}]
        \draw[ultra thick] (0,0) rectangle (1,1);
        \filldraw[ultra thick,fill=gray] (0.4,0.4) rectangle (0.6,0.6);
        \draw[thick] (0.45,0) -- (0.45,0.4);
        \draw[thick] (0.45,0.6) -- (0.45,1);
        \draw (0.5,-0.05) node[anchor=north] {Opposite Lines (OL)};
    \end{scope}
    \end{tikzpicture}
    \caption{The six types of regions which minimize perimeter.}
    \label{fig:sixMinRegions}
\end{figure}

\begin{center}
    \small
    \begin{tabular}{c|l}
        Region & Area Domain\\[5pt]
        \hline\\[-5pt]
        Quarter Circle & $\ds\mathcal{D}_{QC}
        = \left[0,\ \frac{\pi}{8}(1-a)^2\right]$\\[10pt]
        Adjacent Lines & $\ds\mathcal{D}_{AL}
        = \left[\left(\frac{1-a}{2}\right)^2,
        \ \left(\frac{1-a}{2}\right)^2+a(1-a)\right]$\\[10pt]
        Opposite Arcs & $\ds\mathcal{D}_{OA}
        = \left(\left(\frac{1-a}{2}\right)^2+a(1-a),
        \ \left(1+\frac{\pi}{2}\right)\left(\frac{1-a}{2}\right)^2+a(1-a)\right)$\\[10pt]
        Full Line & $\ds\mathcal{D}_{FL}
        = \left[0,\ \frac{1-a}{2}\right)$\\[10pt]
        Adjacent Arcs & $\ds\mathcal{D}_{AA}
        = \left(\frac{1-a}{2}-\frac{\pi}{8}(1-a)^2,
        \ \frac{1-a}{2}\right)$\\[10pt]
        Opposite Lines & $\ds\mathcal{D}_{OL}
        = \left[\frac{1-a}{2},\ \frac{1-a^2}{2}\right]$
        \label{area_domains}
    \end{tabular}
\end{center}

\subsection{Double Arc Regions}
We begin our discussion with the properties of regions bound by two arcs of circles. We refer to these regions as \textit{opposite arc (OA)} and \textit{adjacent arc (AA)} regions. Unfortunately, we cannot express the perimeter of either of these regions explicitly as a function of the area they contain. Hence, we must utilize another variable and implicitly define our perimeter function.
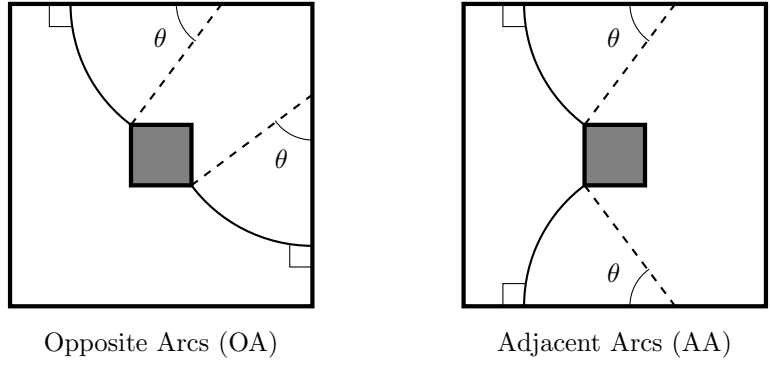
\begin{figure}[ht]
    \centering
    \begin{tikzpicture}[scale=4]
        \begin{scope}[shift={(0,0)}]
            \draw[ultra thick] (0,0) rectangle (1,1);
            \filldraw[ultra thick,fill=gray] (0.4,0.4) rectangle (0.6,0.6);
            \draw[thick] (0.2,1) arc (180:234:0.5);
            \draw[thick] (1,0.2) arc (-90:-144:0.5);
            \draw[thick,dashed] (0.4,0.6) -- (0.7,1);
            \draw[thick,dashed] (0.6,0.4) -- (1,0.7);
            \draw (0.55,1) arc (180:235:0.15);
            \draw (1,0.55) arc (-90:-145:0.15);
            \draw (0.13,1) -- (0.13,0.925) -- (0.205,0.925);
            \draw (1,0.13) -- (0.925,0.13) -- (0.925,0.205);
            \draw (0.55,0.95) node[anchor=north east] {$\theta$};
            \draw (0.95,0.55) node[anchor=north east] {$\theta$};
            \draw (0.5,-0.05) node[anchor=north] {Opposite Arcs (OA)};
        \end{scope}
        \begin{scope}[shift={(1.5,0)}]
            \draw[ultra thick] (0,0) rectangle (1,1);
            \filldraw[ultra thick,fill=gray] (0.4,0.4) rectangle (0.6,0.6);
            \draw[thick] (0.2,1) arc (180:234:0.5);
            \draw[thick] (0.2,0) arc (180:126:0.5);
            \draw[thick,dashed] (0.4,0.6) -- (0.7,1);
            \draw[thick,dashed] (0.4,0.4) -- (0.7,0);
            \draw (0.55,1) arc (180:235:0.15);
            \draw (0.55,0) arc (180:125:0.15);
            \draw (0.13,1) -- (0.13,0.925) -- (0.205,0.925);
            \draw (0.13,0) -- (0.13,0.075) -- (0.205,0.075);
            \draw (0.55,0.95) node[anchor=north east] {$\theta$};
            \draw (0.55,0.05) node[anchor=south east] {$\theta$};
            \draw (0.5,-0.05) node[anchor=north] {Adjacent Arcs (AA)};
        \end{scope}
    \end{tikzpicture}
    \caption{The two types of regions whose boundary consists of two arcs.}
    \label{fig:doubleArcRegions}
\end{figure}

For these regions we introduce the angle $\theta$, and we use the notation $A_{OA}(\theta,a)$ and $A_{AA}(\theta,a)$ to be the area of an opposite arc region and an adjacent arc region, respectively. Similarly, we use $P_{OA}(\theta,a)$ and $P_{AA}(\theta,a)$ to be the perimeter of an opposite arc region and adjacent arc region, respectively. See Figure \ref{fig:doubleArcRegions} for an explanation of which angle $\theta$ is referring to. Now, in terms of $\theta$, we have the following:
\begin{align*}
    A_{OA}(\theta,a) &=  a(1-a) + \left(\frac{1-a}{2}\right)^2[\theta\csc^2\theta - \cot\theta + 1],\\[0.5em]
    A_{AA}(\theta,a) &= \frac{1-a}{2}
    - \left(\frac{1-a}{2}\right)^2[\theta\csc^2\theta - \cot\theta],\\[1em]
    P_{OA}(\theta,a) &= (1 - a)\theta\csc\theta,\\[1em]
    P_{AA}(\theta,a) &= (1 - a)\theta\csc\theta.
\end{align*}

Note that from the geometry, we have that $\theta \in (0,\frac{\pi}{2})$. Now if $A_{OA}$ and $A_{AA}$ are injective, then $P_{OA}$ and $P_{AA}$ are functions of the area they contain. To show this we have the following lemma.

\begin{lemma}\label{lem:Area_Functions_Injective}
    For all fixed $a \in (0,1)$, the functions $A_{OA}(\theta,a)$ and $A_{AA}(\theta,a)$ are injective.
\end{lemma}
\begin{proof}
    Let $a \in (0,1)$ be fixed. Then $A_{OA}(\theta,a)$ and $A_{AA}(\theta,a)$ are only functions of $\theta$. Note that it suffices to show that these functions are monotonic. Taking their derivatives we have
    \[
        \frac{dA_{OA}}{d\theta} = 2\left(\frac{1-a}{2}\right)^2\csc^2\theta\,(1-\theta\cot\theta)
    \]
    and
    \[
        \frac{dA_{AA}}{d\theta} = -2\left(\frac{1-a}{2}\right)^2\csc^2\theta\,(1-\theta\cot\theta).
    \]
    For all $\theta \in (0,\frac{\pi}{2})$, we have $\tan\theta > \theta$. Hence, $1-\theta\cot\theta > 0$. Thus, $A_{OA}$ and $A_{AA}$ are monotonic.
\end{proof}

Thus, for a fixed $a \in (0,1)$, $A_{OA}$ and $A_{AA}$ have inverses. Therefore, the perimeters are in fact functions of the area they enclose. To refer to this we will sometimes write $P_{OA}(A)$ as shorthand for $P_{OA}(A_{OA}^{-1}(A))$, and similarly with $P_{AA}(A)$. Notice that the two area functions have the same derivative with opposite signs. In the next lemma we show that in fact the perimeter functions are reflections of each other.

\begin{lemma}\label{lem:IFL-Symmetry1}
    For all fixed $a \in (0,1)$, the functions $P_{OA}(A)$ and $P_{AA}(A)$ are symmetric about the vertical line, $A=\frac{3}{8}(1-a^2)$.
\end{lemma}
\begin{proof}
    To prove the result, we need
    \[
        P_{OA}(A) = P_{AA}\left(2\cdot \frac{3}{8}(1-a^2) - A\right).
    \] 
    Since $P_{OA} = P_{AA}$, it suffices to show
    \[A_{OA}=2\cdot \frac{3}{8}(1-a^2)-A_{AA}.\]
    So, 
    \begin{align*}
        2\cdot\frac{3}{8}(1-a^2)-A_{AA} &= 2\cdot \frac{3}{8}(1-a^2)
        -\frac{1-a}{2}
        +\left(\frac{1-a}{2}\right)^2(\csc^2\theta-\cot\theta)\\
        &= \frac{3}{4}-\frac{3a^2}{4}-\frac{1}{2}+\frac{a}{2}
        +\frac{1}{4} \left[ (1-a)^2(\theta\csc^2\theta
        -\cot\theta)\right]\\
        &= \frac{1}{4}(1-a)^2+a(1-a)
        +\frac{1}{4}(1-a)^2(\theta\csc^2\theta-\cot\theta)\\
        &= \frac{1}{4}(1-a)^2(\theta\csc^2-\cot\theta+1)+a(1-a)\\
        &= A_{OA}\text{.}
    \end{align*}
    This concludes the proof.
\end{proof}

We now continue determining properties of $P_{OA}(A)$ and $P_{AA}(A)$. The next lemma establishes that $P_{OA}(A)$ as a strictly increasing function, and thus $P_{AA}(A)$ is a strictly decreasing function by Lemma \ref{lem:IFL-Symmetry1}. Afterwards, we determine the limiting behavior of these perimeter functions.

\begin{lemma}\label{lem:IFL-OHincreasing}
    For all fixed $a \in (0,1)$, the function $P_{OA}(A)$ is strictly increasing, and the function $P_{AA}(A)$ is strictly decreasing.
\end{lemma}
\begin{proof}
    Note that the function $P_{OA}$ is only defined when $\theta \in (0,\frac{\pi}{2})$.
    Now by taking the derivative implicitly with respect to area, we have:
    \[
        \frac{dP_{OA}}{dA_{OA}} = \frac{\frac{dP_{OA}}{d\theta}}{\frac{dA_{OA}}{d\theta}} = \frac{(1-a)\csc\theta(1 - \theta\cot\theta)}{\frac{1}{2}(1-a)^2\csc^2\theta(1-\theta\cot\theta)} = \frac{2\sin\theta}{1-a}.
    \]
    Since $\theta \in (0,\frac{\pi}{2})$ and $a \in (0,1)$, we have that $\frac{dP_{OA}}{dA_{OA}}> 0$. By Lemma \ref{lem:IFL-Symmetry1}, this implies that $P_{AA}(A)$ has a strictly negative derivative.
\end{proof}

\begin{lemma}\label{lem:perimeter_limits}
    Let $U_a = \left(\frac{1-a}{2}\right)^2+a(1-a)$ and $V_a = \frac{1-a}{2}$. For all fixed $a \in (0,1)$,
    \[
        \lim_{A \to U_a^+}P_{OA}(A) = 1-a \hspace{2em}\textit{and}\hspace{2em} \lim_{A \to V_a^-}P_{AA}(A) = 1-a.
    \]
\end{lemma}
\begin{proof}
  From Lemma \ref{lem:Area_Functions_Injective}, $A_{OA}(\theta)$ is injective. Further, as $\theta \to 0^+$, we have that $A_{OA}(\theta) \to U_a^+$. Hence,
  \[\lim_{A\to U_a^+}P_{OA}(A) = \lim_{\theta \to 0^+}P_{OA}(\theta) = \lim_{\theta \to 0^+}(1-a)\frac{\theta}{\sin\theta}  = 1-a.\]
  Similarly, as $\theta \to 0^+$, $A_{AA}(\theta) \to V_a^-$. Thus,
  \[
    \lim_{A \to V_a^-}P_{AA}(A) = \lim_{\theta \to 0^+}P_{AA}(\theta) = \lim_{\theta \to 0^+}P_{OA}(\theta) = 1-a.
  \]
\end{proof}

\subsection{Minimizing Region Interactions }
Now that we have established the behavior of the opposite arc and adjacent arc regions, we will make use of this section to analyze the interactions between our six minimizing regions.
Importantly, many of our remaining minimizing regions and their behaviors with one another are dependent on the value of $a$.
We note here that we define the following values $a_3=\frac{4-\pi}{4+3\pi}$ and $a_4=\frac{1}{3}$. We now must work to establish our remaining values of $a$.

\begin{lemma}\label{lemma:IFL-Bisection}
    For all fixed $a \in (0,1)$, the functions $P_{OA}(\theta,a)$ and $P_{AA}(\theta,a)$ are bijective.
\end{lemma}
\begin{proof}
    For a fixed $a$ we have,
    \[P_{OA}(\theta)=P_{AA}(\theta)=(1-a)\theta\csc\theta\]
    then
    \[\frac{dP_{OA}}{d\theta}=\frac{dP_{AA}}{d\theta}
        =(1-a)(\csc\theta-\theta\csc\theta\cot\theta)\text{.}\]
    Since $\tan\theta-\theta>0$ for all $\theta\in(0,\frac{\pi}{2})$, then
    \begin{align*}
        \tan\theta-\theta>0 &\implies \tan\theta>\theta\\
        &\implies 1>\theta\cot\theta\\
        &\implies \csc\theta>\theta\csc\theta\cot\theta\\
        &\implies \csc\theta-\theta\csc\theta\cot\theta>0\\
        &\implies (1-a)(\csc\theta-\theta\csc\theta\cot\theta)>0\\
        &\implies \frac{dP_{OA}}{d\theta}=\frac{dP_{AA}}{d\theta}>0\text{.}
    \end{align*}
    So for all $\theta\in\left(0,\dfrac{\pi}{2}\right)$,
    $P_{OA}$ and $P_{AA}$ are injective.
    Now to show surjectivity, we consider the following limits,
    \[\lim_{\theta\to\frac{\pi}{2}^-}P_{OA}(\theta)
        =\lim_{\theta\to\frac{\pi}{2}^-}P_{AA}(\theta)
        =(1-a)\cdot\lim_{\theta\to\frac{\pi}{2}^-}\theta\csc\theta
        =\dfrac{\pi}{2}(1-a)\text{.}\]
    and from Lemma~\ref{lem:perimeter_limits} we have,
    \[\lim_{\theta\to 0^+}P_{OA}(\theta)
        =\lim_{\theta\to 0^+}P_{AA}(\theta)
        =1-a.\]
    Then by the continuity of $P_{OA}$ and $P_{AA}$ we have that $P_{OA}$ and $P_{AA}$ are surjective.
\end{proof}

The following lemma establishes the existence of a unique intersection point of the profiles of Quarter Circle regions and Opposite Arc regions.

\begin{lemma}\label{lem:mu_exists}
    For all fixed $a \in (0,\ \frac{4-\pi}{4+3\pi})$ there exist unique values $\theta_a \in (0,\ \frac{\pi}{4})$ and $\mu_a\in\left[0,\frac{1}{2}|Q_a|\right]$ such that
    \[
        P_{OA}(\theta_a,a)=\sqrt{\pi\mu_a} \hspace{2em}\textit{and}\hspace{2em}
        \mu_a = A_{OA}(\theta_a,a).
    \]
    Furthermore,
    \begin{itemize}
        \item for all $A \in \left((\frac{1-a}{2})^2+a(1-a),\ \mu_a\right)$, we have $P_{OA}(A) > \sqrt{\pi A}$,
        \item for all $ A \in \left(\mu_a, \frac{\pi}{8}(1-a)^2\right]$, we have $P_{OA}(A) < \sqrt{\pi A}$.
    \end{itemize}
\end{lemma}
\begin{proof}
    Let $a< \frac{4-\pi}{4+3\pi} $. Note that the perimeter of a quarter-circle region is given by $\sqrt{\pi A}$. Thus, for a quarter-circle region and an opposite arc region to have the same the area and perimeter, we must have that
    \[
        P_{OA}(\theta )=\sqrt{\pi A_{OA}(\theta)}.
    \]
    This is true if and only if
    \[
        P_{OA}(\theta)^2 - \pi A_{OA}(\theta) = 0.
    \]
    Now define $D(\theta) := P_{OA}(\theta)^2 - \pi A_{OA}(\theta)$. We will show that $D$ has a unique root in the interval $(0,\frac{\pi}{4})$. To see it has a root within the interval, note that
    \begin{align*}
        \lim_{\theta \to 0^+}D(\theta) &= (1-a)^2 - \pi\left[a(1-a) + \left(\frac{1-a}{2}\right)^2\right] \\
        &= \frac{1}{4}(1-a)\left[4(1-a)-4\pi a - \pi(1-a)\right] \\
        &= \frac{1}{4}(1-a)\left[(4-\pi) - (4+3\pi)a\right] \\
        &> 0,
    \end{align*}
    and
    \begin{align*}
        \lim_{\theta \to \pi/4^-}D(\theta) &= \frac{\pi^2}{8}(1-a)^2 - \pi\left[a(1-a) + \frac{\pi}{8}(1-a)^2\right]\\
        &= -\pi a (1-a)\\
        &< 0.
    \end{align*}
    Now we will show that $D(\theta)$ is monotonic over $(0,\frac{\pi}{4})$. For this, consider
    \begin{align*}
        D'(\theta) &= 2P_{OA}(\theta)P'_{OA}(\theta) - \pi A'_{OA}(\theta)\\
        &= 2(1-a)^2\theta\csc\theta(\csc\theta - \theta\csc\theta\cot\theta) - \frac{\pi}{2}(1-a)^2\csc^2\theta(1-\theta\cot\theta)\\
        &= (1-a)^2\csc\theta\left(2\theta - \frac{\pi}{2}\right)(\csc\theta - \theta\cot\theta).
    \end{align*}
    The first two factors are positive, and the third factor is negative. To see that the final factor is positive, note that
    \[\csc\theta - \theta\cot\theta > 0 \iff \theta\cos\theta < 1.\]
    This is true since $\theta\cos\theta < \frac{\pi}{4}\cos\theta < 1.$ Thus, $D(\theta)$ is monotonically decreasing, and hence it has a unique root $\theta_a \in (0,\frac{\pi}{4})$. This uniquely defines $\mu_a = A_{OA}(\theta_a,a).$

    Let $A \in \left( \left(\frac{1-a}{2}\right)^2+a(1-a), \mu_a\right)$. Since $A_{OA}(\theta)$ is an increasing function, an opposite arc region with area $A$ will have $\theta \in (0,\theta_a)$. Since $D(\theta)$ is monotonically decreasing, we have that within the interval $\theta \in (0,\theta_a)$,
    \begin{align*}
        D(\theta) &> 0 \\
        P_{OA}(\theta)^2 - \pi A_{OA}(\theta) &> 0 \\
        P_{OA}(\theta) &> \sqrt{\pi A_{OA}(\theta)} \\
        P_{OA}(A) &> \sqrt{\pi A}.
    \end{align*}
    Let $A \in \left(\mu_a,\frac{\pi}{8}(1-a)^2\right]$. An opposite arc region with area $A$ corresponds to $\theta \in (\theta_a,\frac{\pi}{4})$. Then we have
    \begin{align*}
        D(\theta) &< 0 \\
        P_{OA}(\theta)^2 - \pi A_{OA}(\theta) &< 0 \\
        P_{OA}(\theta) &< \sqrt{\pi A_{OA}(\theta)} \\
        P_{OA}(A) &< \sqrt{\pi A}.
    \end{align*}
    This concludes the proof.
\end{proof}

Now that we have the uniqueness and existence of $\mu_a$, we must establish a value of $a$ such that this intersection occurs at a perimeter of 1. This shows that at a particular value of $a$, we have a triple intersection between Quarter Circle regions, Opposite Arc regions, and Full Line regions.

\begin{lemma}\label{lem:a1_exists} 
    There exists a unique $a_1 \in (0,\frac{4-\pi}{4+3\pi})$ such that \[\sqrt{\pi \mu_{a_1}} = 1.\]
    Furthermore,
    \begin{itemize}
        \item for all $a \in (0,a_1)$, we have $\sqrt{\pi \mu_a}>1$,
        \item for all $a \in (a_1,\frac{4-\pi}{4+3\pi})$, we have $\sqrt{\pi \mu_a} < 1$.
    \end{itemize}
\end{lemma}
\begin{proof}
    To begin we will first calculate $\mu_a$ at the end points. By \ref{lem:mu_exists}, each $\mu_a$ value uniquely corresponds to a value $\theta_a \in (0,\frac{\pi}{4})$. Since $\mu_a = A_{OA}(\theta_a,a)$ and $P_{OA}(\theta_a,a) = \sqrt{\pi \mu_a}$, we have that
    \begin{equation}\label{eq:a1_exists_1}
        P_{OA}(\theta_a,a)^2 - \pi A_{OA}(\theta_a,a) = 0.
    \end{equation}
    Thus, to find $\mu_a$ as $a$ goes to 0, we need to take the limit of \eqref{eq:a1_exists_1} with $a$ approaching $0$. Doing this, we have the equation
    \[
        \theta_a^2\csc^2\theta_a - \frac{\pi}{4}(\theta_a\csc^2\theta_a - \cot\theta_a + 1) = 0.
    \]
    This equation has the unique solution $\theta_a = \frac{\pi}{4}$. Thus,
    \begin{align*}
        \lim_{a \to 0^+} \mu_a &= \lim_{a \to 0^+}A_{OA}\left(\frac{\pi}{4},a\right)\\
        &= \lim_{a \to 0^+} \frac{(1-a)\left[\pi - a(\pi-8)\right]}{8}\\
        &= \frac{\pi}{8}.
    \end{align*}

    For the other end point, we will again take the limit of \eqref{eq:a1_exists_1}, but now with $a$ approaching $\frac{4-\pi}{4+3\pi}$. Doing this, we have the equation
    \[
        \frac{4\pi^2\left[4(\theta_a^2\csc^2\theta_a-1) - \pi(\theta_a\csc^2\theta_a-\cot\theta_a)\right]}{(4+3\pi)^2} = 0.
    \]
    This equation has no solutions within $(0,\frac{\pi}{4})$. However, taking the limit of the left-hand side as $\theta_a$ approaches $0$ yields $0$. Thus,
    \begin{align*}
        \lim_{a \to \left(\frac{4-\pi}{4+3\pi}\right)^-}\mu_a &= \lim_{\theta_a \to 0^+}A_{OA}\left(\theta_a,\frac{4-\pi}{4+3\pi}\right) \\
        &= \lim_{\theta_a \to 0^+} \frac{4\pi\left[\pi(\theta_a\csc^2\theta_a - \cot\theta_a)+4\right]}{(4+3\pi)^2} \\
        &= \frac{16\pi}{(4+3\pi)^2}.
    \end{align*}

    Now we must calculate the derivative $\frac{d\mu_a}{da}$. For this, note that
    \[
        \mu_a = \frac{1}{\pi}P_{OA}(\theta_a,a)^2 = \frac{1}{\pi}(1-a)^2\theta_a^2\csc^2\theta_a.
    \]
    Taking the derivative with respect to $a$ gives us
    \begin{equation}\label{eq:a1_exists_derivative}
        \frac{d\mu_a}{da} = \frac{2}{\pi}(1-a)\theta_a\csc^2\theta_a\left[(1-\theta_a\csc\theta_a)(1-a)\frac{d\theta_a}{da} - \theta_a\right].
    \end{equation}
    To determine the sign of \eqref{eq:a1_exists_derivative}, we must determine the sign of $\frac{d\theta_a}{da}$. Going back to \eqref{eq:a1_exists_1}, we have
    \[
        (1-a)^2\theta_a^2\csc^2\theta_a - \pi a(1-a) - \frac{\pi}{4}(1-a)^2\left[\theta_a\csc^2\theta_a - \cot\theta_a + 1\right] = 0.
    \]
    Rearranging this equation gives us
    \[
        \theta_a^2\csc^2\theta_a  - \frac{\pi}{4}\left[\theta_a\csc^2\theta_a - \cot\theta_a + 1\right] = \frac{\pi a}{1-a}.
    \]
    Define the left-hand side to be $F(\theta_a)$, and we have
    \[
        F(\theta_a) = \frac{\pi a}{1-a}.
    \]
    Differentiating both sides of this equation with respect to $a$ gives us
    \[
        \frac{dF}{d\theta_a}\cdot \frac{d\theta_a}{d a} = \frac{\pi}{(1-a)^2} > 0.
    \]
    Furthermore,
    \[
        \frac{dF}{d\theta_a} = -\, \frac{(\pi-4\theta_a)\csc^2\theta_a(1-\theta_a\cot\theta_a)}{2}.
    \]
    With $0 < \theta_a < \frac{\pi}{4}$, we have $1-\theta_a\cot\theta_a > 0$. Thus, $\frac{dF}{d\theta_a} < 0$. Therefore, $\frac{d\theta_a}{da} < 0$. Now plugging this back into \eqref{eq:a1_exists_derivative}, we see that $\frac{d \mu_a}{da} < 0$.
    
    Note that $\left(\frac{\pi + 3\pi}{4+3\pi}\right)^2 < 1 < \frac{\pi^2}{8}$. Then dividing by $\pi$, we have
    \[
        \frac{16\pi}{(4+3\pi)^2} < \frac{1}{\pi} < \frac{\pi}{8}.
    \]
    This, along with the fact that $\frac{d\mu_a}{da}<0$, implies that there exists a unique value $a_1 \in (0,\frac{4-\pi}{4+3\pi})$ such that $\mu_{a_1} = \frac{1}{\pi}$. This gives us our first result. Furthermore, the fact that $\mu_a$ is monotonically decreasing gives us the rest of our result.
\end{proof}

We now move on to establishing the intersection between Opposite Arc regions and Full Line regions.

\begin{lemma}\label{lem:gamma_exists}
    For all $a\in(0,\frac{1}{3})$, there exist unique values $\gamma_a\in\mathcal{D}_{OA}$ and $\theta\in(0,\frac{\pi}{2})$ such that $A_{OA}(\theta,a) = \gamma_a$ and $P_{OA}(\theta,a)= 1$.
\end{lemma}
\begin{proof}
    Let $a\in\left(0,\frac{1}{3}\right)$. 
    Recall that
    \[P_{OA}\in\left(1-a,\frac{\pi}{2}(1-a)\right)\text{.}\]
    Since $1-a<1<\frac{\pi}{3}<\frac{\pi}{2}(1-a)$
    then from Lemma~\ref{lemma:IFL-Bisection},
    we have that there exists a unique
    $\theta\in\left(0,\frac{\pi}{2}\right)$
    such that $P_{OA}(\theta,a)=1$.
    Thus we can uniquely define
    $\gamma_a=A_{OA}(\theta,a)\in\mathcal{D}_{OA}$.
\end{proof}

Similarly, we have a unique intersection between Adjacent Arc regions and Full Line regions.

\begin{lemma}\label{lemma:IFL-FLR2IAHIntersection}
    For all $a \in (0,\frac{1}{3})$, there exist unique values $\lambda_a \in \mathcal{D}_{AA}$ and $\theta \in (0,\frac{\pi}{2})$ such that $A_{AA}(\theta,a) = \lambda_a$ and $P_{AA}(\theta,a)= 1$.
\end{lemma}
\begin{proof}
    Let $a\in\left(0,\frac{1}{3}\right)$.
    Recall that
    \[P_{AA}\in\left(1-a,\frac{\pi}{2}(1-a)\right)\text{.}\]
    Since $1-a<1<\frac{\pi}{3}<\frac{\pi}{2}(1-a)$,
    then from Lemma~\ref{lemma:IFL-Bisection},
    we have that there exists a unique
    $\theta\in\left(0,\frac{\pi}{2}\right)$
    such that $P_{AA}(\theta,a)=1$.
    Thus, we can uniquely define
    $\lambda_a=A_{AA}(\theta,a)\in\mathcal{D}_{AA}$.
\end{proof}
\noindent For reference, we used numerical approximations to show the following graphs of $\gamma_a$ and $\lambda_a$, where each grid line represents $0.1$ units.
\begin{figure}[ht]
    \centering
    \begin{tikzpicture}[scale=6]
        \begin{scope}
            \foreach \ind in {0.1,0.2,0.3,0.4} {
                \draw[dotted] (\ind,0.5) -- (\ind,-0.05);
                \draw[dotted] (-0.05,\ind) -- (0.45,\ind);
            };
            \draw[very thick,->] (-0.05,0) -- (0.45,0) node[anchor=north west] {$a$};
            \draw[very thick,->] (0,-0.05) -- (0,0.55) node[anchor=south east]{$\gamma_a$};
            \draw[ultra thick,domain=0:0.375] plot(\x,0.270514027454345 + 2.61094408390567*\x - 22.4231408139832*\x*\x + 147.945860228386*\x*\x*\x - 577.995296353385*\x*\x*\x*\x + 1158.46040891728*\x*\x*\x*\x*\x - 924.331691421317*\x*\x*\x*\x*\x*\x);
        \end{scope}
        \begin{scope}[shift={(0.8,0)}]
            \foreach \ind in {0.1,0.2,0.3,0.4} {
                \draw[dotted] (\ind,0.5) -- (\ind,-0.05);
                \draw[dotted] (-0.05,\ind) -- (0.45,\ind);
            };
            \draw[very thick,->] (-0.05,0) -- (0.45,0) node[anchor=north west] {$a$};
            \draw[very thick,->] (0,-0.05) -- (0,0.55) node[anchor=south east]{$\lambda_a$};
            \draw[ultra thick,domain=0:0.375] plot(\x,0.4795-2.6109*\x+21.67314*\x*\x - 147.94586*\x*\x*\x + 577.9953*\x*\x*\x*\x - 1158.4604*\x*\x*\x*\x*\x + 924.3317*\x*\x*\x*\x*\x*\x);
        \end{scope}
    \end{tikzpicture}
    \caption{The constants $\gamma_a$ and $\lambda_a$}
    \label{fig:gamma_and_lambda}
\end{figure}
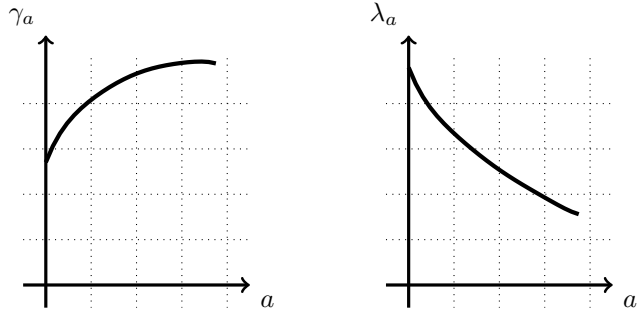
Now that these intersection areas $\gamma_a$ and $\lambda_a$ have been established, we must explore some of their properties. In Lemma \ref{lemma:IFL-Change}, we show how $\gamma_a$ grows, and in Lemma \ref{lem:Lambda_Greater_1overPi} we provide a lower bound for $\lambda_a$.

\begin{lemma}\label{lemma:IFL-Change}
    For all $a \in (0,\frac{1}{3})$, we have that
    $\frac{d\gamma_a}{da} > 0$.
\end{lemma}
\begin{proof}
    Let $a\in (0\ ,\frac{1}{3})$.
    Recall that $A_{OA}$ is only defined for
    $\theta \in (0,\frac{\pi}{2})$.
    Then, by Lemma~\ref{lem:gamma_exists} we have that $\gamma_a=A_{OA}(\theta,a)\in\mathcal{D}_{OA}$,
    where $P_{OA}(\theta,a)=(1-a)\theta\csc\theta=1$. Solving the perimeter equation for $a$ and substituting it into $A_{OA}$ gives us $\gamma_a$ as a function of $\theta$. Now by differentiating we have
    \[\frac{d\gamma_a}{da}=\frac{\frac{d\gamma_a}{d\theta}}{\frac{da}{d\theta}}=\frac{3\sin\theta+\cos\theta - 2\theta}{2\theta}.\]
    We can see that $\frac{d\gamma_a}{da} > 0$ with the following argument. When $a < \frac{1}{3}$ and $0<\theta < \frac{\pi}{2}$, we have:
    \begin{align*}
        (1-a)\theta\csc\theta=1
        &\implies \theta\csc\theta=\frac{1}{1-a} < \frac{3}{2}\\
        &\implies 0<3\sin\theta-2\theta < 3\sin\theta + \cos\theta - 2\theta = \frac{d\gamma_a}{da}.
    \end{align*}
\end{proof}

\begin{lemma}\label{lem:Lambda_Greater_1overPi}
    For all $a \in (0,1-\frac{2}{\pi})$, $\lambda_a > \frac{1}{\pi}$.
\end{lemma}
\begin{proof}
    By definition, $\lambda_a$ occurs when we have $(1-a)\theta\csc\theta = 1$. Solving this for $a$ and substituting into $A_{AA}(\theta,a)$ gives us $\lambda_a$ as a function of $\theta$. We will calculate the second derivative of $\lambda_a$ with respect to $a$.
    \[
        \frac{d^2\lambda_a}{da^2} = \frac{\frac{d^2\lambda_a}{d\theta^2}\frac{da}{d\theta} - \frac{d\lambda_a}{d\theta}\frac{d^2a}{d\theta^2}}{\left(\frac{da}{d\theta}\right)^3} = \frac{\theta\sin\theta + \cos\theta}{\theta\cos\theta - \sin\theta}.
    \]
    Since $\theta \in (0,\frac{\pi}{2})$, we have that the numerator is always positive, and the denominator is always negative. Hence, $\lambda_a$ is concave down. Thus, its minimum occurs at an end point. Also, we have that
    \[
        \lim_{a \to 0^+} \lambda_a = \lim_{\theta\to 0^+}A_{AA}(\theta,0) = \frac{1}{2} > \frac{1}{\pi},
    \]
    and
    \[
        \lim_{a \to (1-\frac{2}{\pi})^-} \lambda_a = A_{AA}\left(\frac{\pi}{2},1-\frac{2}{\pi}\right) = \frac{3}{2\pi} > \frac{1}{\pi}.
    \]
    This concludes the proof.
\end{proof}

Now that we know more about these intersection points, we can establish a value of $a$ for which another triple intersection occurs. This time Opposite Arc regions, Adjacent Arc regions, and Full Line regions all intersect.

\begin{lemma}\label{lemma:IFL-Unique_a2}
    There exist unique values $a_2 \in (0,1)$
    and $\theta \in (0,\frac{\pi}{2})$
    such that $A_{OA}(\theta,a_2) = A_{AA}(\theta,a_2)$,
    and $P_{OA}(\theta,a_2)=P_{AA}(\theta,a_2) = 1$.
\end{lemma}
\begin{proof}
    Since
    \begin{align*}
        P_{OA}(\theta,a_2)=P_{AA}(\theta,a_2)=1 &\implies
        (1-a_2)\theta\csc\theta=1\\
        &\implies a_2 = 1-\frac{\sin\theta}{\theta}
    \end{align*}
    then $a_2$ is unique if and only if there is a unique $\theta \in (0,\frac{\pi}{2})$ which solves $A_{OA}(\theta,1-\frac{\sin\theta}{\theta}) = A_{AA}(\theta,1-\frac{\sin\theta}{\theta})$. By substitution we reach a solution to the equation when $\theta$ is a zero of the following function.
    Let
    \begin{align*}
        F(\theta) &= A_{OA}\left(\theta,1-\frac{\sin\theta}{\theta}\right)
        -A_{AA}\left(\theta,1-\frac{\sin\theta}{\theta}\right)\\
        &=\frac{2\theta + 2\theta\sin\theta
        -2\sin\theta\cos\theta-3\sin^2\theta}{4\theta^2}.
    \end{align*}
    Differentiating $F$ gives us
    \[F'(\theta) = \frac{(\theta -2\cos\theta - 3\sin\theta)(\theta\cos\theta-\sin\theta)}{2\theta^3}.\]
    We now need to show that $F'(\theta)$ is always positive. Note that the denominator of $F'$ is positive.
    We can see that $\theta\cos\theta - \sin\theta < 0$ since
    \[\lim_{\theta \to 0^+}(\theta\cos\theta - \sin\theta) = 0\]
    and
    \[\frac{d}{d\theta}(\theta\cos\theta - \sin\theta) = -\theta\sin\theta<0    \text{ for all }\theta \in \left(0,\frac{\pi}{2}\right).\]
    Since for all $\theta \in (0,\frac{\pi}{2})$,
    \[\frac{d^2}{d\theta^2}(\theta - 2\cos\theta - 3\sin\theta) =               2\cos\theta + 3\sin\theta>0\]
    and $\theta - 2\cos\theta - 3\sin\theta<0$ at $\theta = 0$ and $\theta = \frac{\pi}{2}$. 
    Then $\theta - 2\cos\theta - 3\sin\theta<0$ for all $\theta \in (0,\frac{\pi}{2})$.
    Thus $F'(\theta)>0$.
    Now notice that
    \[\lim_{\theta \to 0^+}F(\theta) = -\frac{1}{4} < 0\]
    and
    \[\lim_{\theta\to\frac{\pi}{2}^-}F(\theta) = \frac{2\pi-3}{\pi^2}>0.\]
    Since $F$ has opposite signs on its end points and is always increasing.
    Then $F$ has a unique zero in the interval $(0,\frac{\pi}{2})$.
\end{proof}

These values where triple intersections occur will be useful for us in the next section. Each time we encounter one of these values, our isoperimetric profile will change. Thus, we will have five different profile types separated by the four values of $a$:
\[
    a_1 \approx 0.0212, \hspace{1em} a_2 \approx0.0623, \hspace{1em} a_3=\frac{4-\pi}{4+3\pi}\approx0.06394, \hspace{1em} a_4 = \frac{1}{3}.
\]

\section{Proof of main theorem}\label{sec:MainTheorem}
Before providing a proof of the main theorem, we need a few more tools. The next four lemmas will make the proofs of the theorems simpler.

\begin{lemma}\label{lemma:PCL-QC_First}
   For all $a \in (0,1)$ and for all
   $A \in \left[0,\ \frac{(1-a)^2}{\pi}\right]$,
   $P_{QC}$ is a minimizing shape.
\end{lemma}
\begin{proof}
    Since $A \in \left[0,\ \frac{(1-a)^2}{\pi}\right]$ then
    \begin{align*}
        A\leq\dfrac{(1-a)^2}{\pi}
        &\implies\sqrt{\pi A}\leq1-a\\
        &\implies P_{QC}(A)\leq1-a
    \end{align*}
    and since the perimeters of all other minimizing shapes are bounded below by $1-a$ then $P_{QC}$ is a minimizing shape.
\end{proof}

\begin{lemma}\label{lemma:PCL-QC<ASP}
    For all $a\in\left(0,\ \frac{4-\pi}{4+3\pi}\right]$,
    $P_{QC}(A)\leq P_{AL}(A)$.
\end{lemma}
\begin{proof}
  When $a\leq\frac{4-\pi}{4+3\pi}$ then we have,
  \begin{align*}
    a(4+3\pi) &\leq 4-\pi\\
    4a+3\pi a &\leq 4-\pi\\
    4a+4\pi a-\pi a &\leq 4-\pi\\
    4\pi a+a(4-\pi) &\leq 4-\pi\\
    4\pi a &\leq 4-\pi-a(4-\pi)\\
    4\pi a &\leq (4-\pi)(1-a)\\
    \frac{a}{1-a} &\leq \frac{4-\pi}{4\pi}\\
    \frac{a}{1-a} &\leq \frac{1}{\pi}-\frac{1}{4}\\
    \frac{1}{4}+\frac{a}{1-a} &\leq \frac{1}{\pi}\\
    \frac{1}{4}(1-a)^2+a(1-a) &\leq \frac{1}{\pi}(1-a)^2\text{.}
  \end{align*}
  Let
  $A\in\left[\frac{(1-a)^2}{4},
  \frac{(1-a)^2}{4}+a(1-a)\right]=\mathcal{D}_{AL}$
  then $A\leq\frac{(1-a)^2}{\pi}$.
  Thus
  \[\sqrt{\pi A}\leq\sqrt{\pi\cdot\dfrac{(1-a)^2}{\pi}}=1-a\]
  and so $P_{QC}(A)\leq P_{AL}(A)$.
\end{proof}

\begin{lemma}\label{lemma:PCL-ASP_Minimizer}
  For all $a\in\left[\frac{4-\pi}{4+3\pi},1\right)$ and for all
  \[A\in\left[\dfrac{(1-a)^2}{\pi},\dfrac{(1-a)^2}{4}+a(1-a)\right],\]
  $P_{AL}$ is a minimizing shape.
\end{lemma}
\begin{proof}
  Let $a\geq\frac{4-\pi}{4+3\pi}$ then we have,
  \[\frac{(1-a)^2}{\pi}\leq\frac{(1-a)^2}{4}+a(1-a)\text{.}\]
  Since $\left[\frac{(1-a)^2}{\pi},\frac{(1-a)^2}{4}+a(1-a)\right]$ is a non-empty subset of
  \[\left[\dfrac{(1-a)^2}{4},
  \dfrac{(1-a)^2}{4}+a(1-a)\right]=\mathcal{D}_{AL},\]
  we then let $A\in\left[\frac{(1-a)^2}{\pi},\frac{(1-a)^2}{4}+a(1-a)\right]$.
  Then
  \[A\geq\dfrac{(1-a)^2}{\pi} \implies \sqrt{\pi A}\geq1-a\]
  so $P_{QC}(A)\geq P_{AL}(A)$ and since for all $a\in(0,1)$ we have $P_{AL}=1-a$, we then compare
  \begin{itemize}
    \item $P_{FL}=1>1-a$
    \item $P_{AA}>1-a$
    \item $P_{OL}=1-a$.
  \end{itemize}
  Therefore $P_{AL}$ is a minimizing shape.
\end{proof}

\begin{lemma}\label{lemma:PCL-OSP_Final}
    For all $a\in (0,1)$ and for all $A\in\left[\frac{1-a}{2},\frac{1-a^2}{2}\right]$, $P_{OL}$ is the minimizing shape.
\end{lemma}
\begin{proof}
    Let $a\in (0,1)$ and
    $A\in\left[\frac{1-a}{2},\frac{1-a^2}{2}\right]$.
    Then $A\not\in\mathcal{D}_{FL}$ and $A\not\in\mathcal{D}_{AA}$.
    Since $P_{OL}=P_{AL}<P_{OA}$ then $P_{OA}$ is not a minimizing shape and for any area bounded by $P_{AL}$, we can bound the same area with $P_{OL}$ without increasing perimeter.
    Now for $P_{QC}(A)\leq P_{OL}(A)$ then $A\leq\frac{(1-a)^2}{\pi}$
    and for
    \[\frac{(1-a)^2}{\pi}\in\left[\frac{1-a}{2},\frac{1-a^2}{2}\right]\]
    then
    \[\frac{1-a}{2}\leq\frac{(1-a)^2}{\pi}
        \implies a\leq1-\frac{\pi}{2}<0.\]
    Thus $P_{QC}$ is not a minimizing shape.
    Therefore, $P_{OL}$ is a minimizing shape.
\end{proof}

From here we will prove the main theorem. As previously stated, our main theorem provides the relative isoperimetric profile $f_a:[0,\frac{1}{2}|Q_a|] \to \mathbb{R}$ for all possible $a \in (0,1)$. Theorems \ref{theorem:Region_5} through \ref{theorem:Region_1} show that the profiles provides are indeed minimal within a particular range of $a$-values.

\begin{theorem}\label{theorem:Region_5}
    For all $a \in (0,a_1)$, the relative isoperimetric profile function, $f_a:\left[0,\frac{1}{2}|Q_a|\right] \to \mathbb{R}$, is given by
    \begin{equation*}
        f_a(A)=
        \left\{
        \begin{array}{c@{\,,\quad}c@{\;}c@{\;}c@{\;}c@{\;}c@{\quad-\quad}l}
            \sqrt{\pi A}
            & 0 & \leq & A & \leq & \dfrac{1}{\pi}
            & \textnormal{Quarter Circle}\\[0.6em]

            1
            & \dfrac{1}{\pi} & \leq & A & \leq & \lambda_a
            & \textnormal{Full Line}\\[0.6em]

            P_{AA}(A)
            & \lambda_a & \leq & A & < & \dfrac{1-a}{2}
            & \textnormal{Adjacent Arcs}\\[0.6em]

            1-a
            & \dfrac{1-a}{2} & \leq & A & \leq & \dfrac{1}{2}|Q_a|
            & \textnormal{Opposite Lines.}
        \end{array}
        \right.
    \end{equation*}
\end{theorem}
\begin{proof}
    By Lemma \ref{lemma:PCL-QC<ASP}, no Adjacent Line regions will appear. Next we will eliminate the possibility of Opposite Arc regions. Let $A \in \mathcal{D}_{OA}$. If $A \leq \mu_a$, then by Lemma \ref{lem:mu_exists}, $P_{OA}(A) \geq \sqrt{\pi A} = P_{QC}(A)$. If $A > \mu_a$, then by Lemma \ref{lem:a1_exists} we have
    \[
        \sqrt{\pi A} > \sqrt{\pi \mu_a} > 1 = P_{FL}(A).
    \]
    Thus, no Opposite Arc region will appear.
    
    It suffices to find the minimum of the Quarter Circle, Full Line, Adjacent Arcs, and Opposite Lines regions. Begin by assuming $A \in [0,\frac{1}{\pi}]$. Then $A$ is not in $\mathcal{D}_{OL}$, and thus Opposite Line regions will not be considered. By Lemma \ref{lem:Lambda_Greater_1overPi}, $\lambda_a > \frac{1}{\pi}$, and therefore by Lemma \ref{lem:IFL-OHincreasing}, $P_{AA}(A) > 1$. Thus, Adjacent Arc regions will not be considered. Finally, since $A \leq \frac{1}{\pi}$, we have
    \[
        \sqrt{\pi A} \leq 1 = P_{FL}(A).
    \]
    Hence, Quarter Circle regions are the only option.
    
    Now assume that $A \in [\frac{1}{\pi},\lambda_a]$. This means $\sqrt{\pi A} \geq 1 = P_{FL}(A)$. So Quarter Circle regions are not minimum. From Lemma \ref{lem:IFL-OHincreasing}, we have that $P_{AA}(A) \geq 1 = P_{FL}(A)$. So Adjacent Arc regions are not minimum. Finally, $\lambda_a < \frac{1-a}{2}$, and thus an area of $A$ cannot be bound by Opposite Lines. Hence, Full Line regions are the only option.
    
    Assume that $A \in [\lambda_a,\frac{1-a}{2})$. Again, for these $A$ values, $\sqrt{\pi A} > 1$, eliminating Quarter Circle regions. From Lemma \ref{lem:IFL-OHincreasing} and from the definition of $\lambda_a$, we have $P_{AA}(A) \leq 1 = P_{FL}(A)$. This eliminates Full Line regions. Finally, $A \notin \mathcal{D}_{OL}$. Hence, Adjacent Arc regions are the only option.
    
    Assume that $A \in [\frac{1-a}{2},\frac{1}{2}|Q_a|]$. By Lemma \ref{lemma:PCL-OSP_Final}, Opposite Line regions are minimum. This concludes the proof.
\end{proof}

\begin{theorem}\label{theorem:Region_4}
    For all $a \in [a_1,a_2)$, the relative isoperimetric profile function, $f_a:\left[0,\frac{1}{2}|Q_a|\right] \to \mathbb{R}$, is given by
    \begin{equation*}
        f_a(A)=
        \left\{
        \begin{array}{c@{\,,\quad}c@{\;}c@{\;}c@{\;}c@{\;}c@{\quad-\quad}l}
            \sqrt{\pi A}
            & 0 & \leq & A & \leq & \mu_a
            & \textnormal{Quarter Circle}\\[0.6em]
    
            P_{OA}(A)
            & \mu_a & \leq & A & \leq & \gamma_a
            & \textnormal{Opposite Arcs}\\[0.6em]
    
            1
            & \gamma_a & \leq & A & \leq & \lambda_a
            & \textnormal{Full Line}\\[0.6em]
    
            P_{AA}(A)
            & \lambda_a & \leq & A & < & \dfrac{1-a}{2}
            & \textnormal{Adjacent Arcs}\\[0.6em]
    
            1-a
            & \dfrac{1-a}{2} & \leq & A & \leq & \dfrac{1}{2}|Q_a|
            & \textnormal{Opposite Lines}.
        \end{array}
        \right.
    \end{equation*}
\end{theorem}
\begin{proof}
    From Lemma \ref{lemma:PCL-QC<ASP}, no Adjacent Arc regions will appear. Thus, it suffices to find the minimum of the five remaining regions.

    Let $A \in [0,\mu_a]$. By Lemmas  \ref{lem:mu_exists} and \ref{lem:a1_exists}, we have that $\sqrt{\pi A} \leq P_{OA}(A)$. Thus, no Opposite Arc regions will appear. From Lemma \ref{lem:a1_exists}, we have that
    \[
        P_{QC}(A) = \sqrt{\pi A} \leq \sqrt{\pi \mu_a} < 1 = P_{FL}(A).
    \]
    Thus, Full Line regions are eliminated. Since $\mu_a < \lambda_a$, then by Lemma \ref{lem:IFL-OHincreasing} $P_{AA}(A) > 1$. So no Adjacent Arc regions need to be considered. Finally, $A \notin \mathcal{D}_{OL}$. Hence, Quarter Circle regions are minimum regions.

    Let $A \in [\mu_a, \gamma_a]$. By Lemma \ref{lem:mu_exists}, we have $P_{OA}(A) \leq \sqrt{\pi A} = P_{QC}(A)$, which eliminates Quarter Circle regions. By Lemmas \ref{lem:IFL-OHincreasing} and \ref{lem:gamma_exists}, we have $P_{OA}(A) \leq 1 = P_{FL}(A)$, which eliminates Full Line regions. By Lemma \ref{lemma:IFL-Change}, $A < \lambda_a$. Then by Lemma \ref{lem:IFL-OHincreasing}, $P_{AA}(A) > 1$. This eliminates Adjacent Arc regions. Finally, $A \notin \mathcal{D}_{OL}$. Hence, Opposite Arc regions are minimum.

    Let $A \in [\gamma_a,\lambda_a]$. From Lemmas \ref{lem:IFL-OHincreasing} and \ref{lem:gamma_exists}, we have that $P_{OA}(A) \geq 1$, which eliminates Opposite Arc regions. By Lemma \ref{lem:mu_exists}, we have that $P_{OA}(A) < P_{QC}(A)$. This also eliminates Quarter Circle regions. From Lemmas \ref{lem:IFL-OHincreasing} and \ref{lemma:IFL-FLR2IAHIntersection}, we have that $P_{AA}(A) \geq 1$, which eliminates Adjacent Arc regions. Finally, $A \notin \mathcal{D}_{OL}$. Hence, Full Line regions are minimum.

    Let $A \in [\lambda_a, \frac{1-a}{2})$. By Lemma \ref{lem:Lambda_Greater_1overPi}, $\lambda_a > \frac{1}{\pi}$. Thus,
    \[
        P_{QC}(A) = \sqrt{\pi A} \geq \sqrt{\pi \lambda_a} > \sqrt{\pi \cdot \frac{1}{\pi}} = 1.
    \]
    This eliminates Quarter Circle regions. Since $\lambda_a \geq \gamma_a$, then by Lemmas \ref{lem:IFL-OHincreasing} and \ref{lem:gamma_exists}, we have $P_{OA}(A) > 1$. This eliminates Opposite Arc regions. By Lemmas \ref{lem:IFL-OHincreasing} and \ref{lemma:IFL-FLR2IAHIntersection}, we have $P_{AA}(A) \leq 1 = P_{FL}(A)$. Thus, Full Line regions are eliminated. Finally, $A \notin \mathcal{D}_{OL}$. Hence, Adjacent Arc regions are minimum.

    Let $A \in [\frac{1-a}{2},\frac{1}{2}|Q_a|]$. By Lemma \ref{lemma:PCL-OSP_Final}, Opposite Line regions are minimum. This concludes the proof.
\end{proof}

\begin{theorem}\label{theorem:Region_3}
   For all $a \in [a_2,\frac{4-\pi}{4+3\pi})$, the relative isoperimetric profile function, $f_a:\left[0,\frac{1}{2}|Q_a|\right] \to \mathbb{R}$, is given by
    \begin{equation*}
        f_a(A)=
        \left\{
        \begin{array}{c@{\,,\quad}c@{\;}c@{\;}c@{\;}c@{\;}c@{\quad-\quad}l}
            \sqrt{\pi A}
            & 0 & \leq & A & \leq & \mu_a
            & \textnormal{Quarter Circle}\\[0.6em]
    
            P_{OA}(A)
            & \mu_a & \leq & A & \leq & \dfrac{3}{8}(1-a^2)
            & \textnormal{Opposite Arcs}\\[0.6em]
    
            P_{AA}(A)
            & \dfrac{3}{8}(1-a^2) & \leq & A & < & \dfrac{1-a}{2}
            & \textnormal{Adjacent Arcs}\\[0.6em]
    
            1-a
            & \dfrac{1-a}{2} & \leq & A & \leq & \dfrac{1}{2}|Q_a|
            & \textnormal{Opposite Lines}.
        \end{array}
        \right.
    \end{equation*}
\end{theorem}
\begin{proof}
    From Lemma \ref{lemma:PCL-QC<ASP}, we know that no Adjacent Line regions will appear. Assume that $A \in [0,\mu_a]$. Within this interval of areas, we must find the minimum of Quarter Circle regions, Full Line regions, and Opposite Arc regions. Then from Lemma \ref{lem:mu_exists}, we have that $\sqrt{\pi A} < P_{OA}(A)$, which eliminates Opposite Arc regions. From Lemma \ref{lem:a1_exists}, we have that $\sqrt{\pi A} \leq \sqrt{\pi \mu_a} < 1$, which eliminates Full Line regions. Hence, Quarter Circle regions must be minimum.

    Let $A \in [\mu_a, \frac{3}{8}(1-a)^2]$. From Lemma \ref{lem:mu_exists}, we have that $P_{OA}(A) < \sqrt{\pi A}$, which eliminates Quarter Circle regions. By the symmetry in Lemma \ref{lem:IFL-Symmetry1} and by Lemma \ref{lem:IFL-OHincreasing}, we have that $P_{OH}(A) \leq P_{AA}(A)$. This eliminates Adjacent Arc regions. Note that by Lemma \ref{lem:IFL-Symmetry1}$, \gamma_{a_2} = \frac{3}{8}(1-a^2)$. Then by Lemma \ref{lemma:IFL-Change}, for all $a \in [a_2,\frac{1}{3})$, we have that $\gamma_a \geq \frac{3}{8}(1-a^2)$. Since $P_{OA}(\gamma_a)=1$ by definition, and since $P_{OA}(A)$ is always increasing, we know that $P_{OA}(A) \leq 1$ in this area region. This eliminates the possibility of Full Line regions. Finally, $A \notin \mathcal{D}_{OL}$. Hence, Opposite Arc regions are minimum.
    
    Let $A \in [\frac{3}{8}(1-a^2), \frac{1-a}{2})$. From the symmetry established in Lemma \ref{lem:IFL-Symmetry1}, we have that $P_{AA}(A) \leq P_{OA}(A)$, eliminating Opposite Arc regions. This symmetry also gives us that $P_{AA}(A) \leq 1$, which eliminates Full Line regions. Finally, $A \notin \mathcal{D}_{QC}$ and $A \notin \mathcal{D}_{OL}$. Hence, Adjacent Arc regions are minimum.

    Let $A \in [\frac{1-a}{2},\frac{1}{2}|Q_a|]$. By Lemma \ref{lemma:PCL-OSP_Final}, Opposite Line regions are minimum. This concludes the proof.
\end{proof}

\begin{theorem}\label{theorem:Region_2} 
    For all $a \in [\frac{4-\pi}{4+3\pi},\frac{1}{3})$, the relative isoperimetric profile function, $f_a:\left[0,\frac{1}{2}|Q_a|\right] \to \mathbb{R}$, is given by
    \begin{equation*}
        f_a(A)=
        \left\{
        \begin{array}{c@{\,,\quad}c@{\;}c@{\;}c@{\;}c@{\;}c@{\quad-\quad}l}
            \sqrt{\pi A}
            & 0 & \leq & A & \leq & \dfrac{(1-a)^2}{\pi} &
            \textnormal{Quarter Circle}\\[0.6em]
            1-a
            & \dfrac{(1-a)^2}{\pi} & \leq & A & \leq & K &
            \textnormal{Adjacent Lines}\\[0.6em]
            P_{OA}(A)
            & K & \leq & A & \leq & \dfrac{3}{8}(1-a^2) &
            \textnormal{Opposite Arcs}\\[0.6em]
            P_{AA}(A)
            & \dfrac{3}{8}(1-a^2) & \leq & A & \leq & \dfrac{1-a}{2} &
            \textnormal{Adjacent Arcs}\\[0.6em]
            1-a
            & \dfrac{1-a}{2} & \leq & A & \leq & \dfrac{1}{2}|Q_a| &
            \textnormal{Opposite Lines},
        \end{array}
        \right.
    \end{equation*}
    where $K = \left(\frac{1-a}{2}\right)^2 + a(1-a)$.
\end{theorem}
\begin{proof}
    If $A \in [0,\frac{1}{\pi}(1-a)^2]$, then from Lemma \ref{lemma:PCL-QC_First}, we have that $P_{QC}(A)=\sqrt{\pi A}$ is a minimizer. If instead $A \in [\frac{1}{\pi}(1-a), K]$, then by Lemma \ref{lemma:PCL-ASP_Minimizer}, we have that $P_{AL}(A) = 1-a$ is a minimizer.

    Now suppose $A \in [K,\frac{3}{8}(1-a^2)]$. Within these area bounds we only need to consider Opposite Arc regions, Adjacent Arc regions, and Full Line regions. From the symmetry established in Lemma \ref{lem:IFL-Symmetry1}, and by utilizing the same arguments as in Theorem \ref{theorem:Region_3}, we have that Opposite Arc regions are minimal. Similarly, if $A \in [\frac{3}{8}(1-a^2),\frac{1}{2}(1-a)]$, then Adjacent Arc regions are minimal.

    Finally, suppose $A \in [\frac{1}{2}(1-a),\frac{1}{2}|Q_a|]$. Then by Lemma \ref{lemma:PCL-OSP_Final}, Opposite Line regions are minimum. This concludes the proof.
\end{proof}

\begin{theorem}\label{theorem:Region_1}
    For all $a \in [\frac{1}{3},1)$, the relative isoperimetric profile function, $f_a:\left[0,\frac{1}{2}|Q_a|\right] \to \mathbb{R}$, is given by
    \begin{equation*}
        f_a(A)=
        \left\{
        \begin{array}{c@{\,,\quad}c@{\;}c@{\;}c@{\;}c@{\;}c@{\quad-\quad}l}
            \sqrt{\pi A}
            & 0 & \leq & A & \leq & \dfrac{(1-a)^2}{\pi} &
            \textnormal{Quarter Circle}\\[0.6em]
            1-a
            & \dfrac{(1-a)^2}{\pi} & \leq & A & \leq & \frac{1}{2}|Q_a| &
            \textnormal{Adjacent/Opposite Lines}.
        \end{array}
        \right.
    \end{equation*}
\end{theorem}
\begin{proof}   
    If $A\in [0\ ,\frac{1}{\pi}(1-a)^2]$, then by Lemma \ref{lemma:PCL-QC_First}, $P_{QC}(A) = \sqrt{\pi A}$ is minimal. If on the other hand $A \in [\frac{1}{\pi}(1-a)^2,\frac{1}{2}|Q_a|]$, then we have
    \[
        \frac{(1-a)^2}{4} < A \leq \frac{1}{2}|Q_a|.
    \]
    Furthermore, since $a \geq \frac{1}{3}$, we have that $\mathcal{D}_{AL} \cup \mathcal{D}_{OL}$ is a connected interval. Hence, $P_{AL}(A) = P_{OL}(A) = 1-a$ is minimal.
\end{proof}

\newpage

\section*{Data Availability}
No data was used for the research described in this article.

\section*{Declaration of Competing Interest}
The authors state that there are no conflicts of interest.

\section*{Acknowledgments}
This work was conducted as part of the Mathematical Association of America's National Research Experiences for Undergraduates Program (NREUP) at the University of Tennessee at Martin, supported by funding from the National Science Foundation under award DMS-2308688. The authors also thank the program director, Dr. Jason DeVito, the Department of Mathematics and Statistics, and the Office of Research and Sponsored Programs at the University of Tennessee at Martin for the resources provided during this research.

\bibliographystyle{plain}
\bibliography{biblio}
\end{document}